\documentclass[a4paper,11pt]{amsart}
\usepackage[utf8]{inputenc}
\usepackage{epsfig,graphics}
\usepackage[all,knot,poly,color,curve]{xy}
\usepackage{array}
\usepackage{amsthm,amssymb,amsbsy,bbm,a4wide,verbatim,url}
\usepackage[marginpar]{todo}
\usepackage{longtable}
\usepackage{rotating}
\usepackage{color,fdsymbol}
\usepackage{hyperref}

\usepackage{fancybox}
\usepackage{enumitem}

\newcommand{\rvline}{\hspace*{-\arraycolsep}\vline\hspace*{-\arraycolsep}}

\usepackage{tikz,tkz-euclide}
\usepackage{tikz-cd} 
\usepackage{tikz-3dplot}
\usetikzlibrary{arrows, decorations.text, decorations.pathmorphing, decorations.markings, positioning, shapes}
\usetikzlibrary{}
\usetikzlibrary{calc}
\usetikzlibrary{decorations.markings}

\DeclareMathOperator{\rk}{\mathrm{rk}}

\def\Id{\mathrm{Id}}

\def\un{\mathbbm{1}}

\def\into{\hookrightarrow}
\def\onto{\twoheadrightarrow}

\def\om{\omega}

\def\kk{\mathbbm{k}}

\def\C{\ensuremath{\mathbbm{C}}}
\def\Q{\mathbbm{Q}}

\def\Z{\mathbbm{Z}}

\def\sl{\mathfrak{sl}}

\newcommand{\la}{\lambda}

\def\eps{\varepsilon}
\def\ii{\mathrm{i}}
\def\dd{\mathrm{d}}

\def\Ker{\mathrm{Ker}}

\def\End{\mathrm{End}}

\def\GL{\mathrm{GL}}

\def\QQ0{\ensuremath{Q_0}}

\def\KK{\ensuremath{K}}

\newtheorem{theorem}{Theorem}[section]
\newtheorem{lemma}[theorem]{Lemma}

\newtheorem{remark}[theorem]{Remark}

\newtheorem{corollary}[theorem]{Corollary}

\newtheorem{example}{Example\/}

\newtheorem{proposition}[theorem]{Proposition}

\numberwithin{equation}{section}

\date{October 22, 2020}

\title[Hecke algebras of normalizers of parabolic subgroups]{Hecke algebras of normalizers of parabolic subgroups}

\author{Thomas Gobet}
\address{Institut Denis Poisson, CNRS UMR 7350\\
Facult\'e des Sciences et Techniques\\
Universit\'e de Tours\\
Parc de Grandmont\\
37200 Tours, France}
\email{thomas.gobet@univ-tours.fr}

\author{Ivan Marin}
\address{Laboratoire Ami\'enois de Math\'ematique Fondamentale et Appliqu\'ee, CNRS UMR 7352\\
Universit\'e de Picardie Jules Verne\\
33 rue Saint Leu\\
80039 Amiens, France}
\email{ivan.marin@u-picardie.fr}

\begin{document}

\maketitle

\begin{abstract}
In the context of Hecke algebras of complex reflection groups, we prove that the generalized Hecke algebras of normalizers of parabolic subgroups are
semidirect products, under suitable conditions on the parameters involved in their definition. 
\end{abstract}

\tableofcontents

\section{Introduction}

Let $W$ be a complex reflection group, that is, a finite subgroup of $\GL_n(\C)$ generated by complex
(pseudo-)reflections. Let $W_0$ be a parabolic subgroup of $W$, that is, the pointwise stabilizer
of a subset of $\C^n$, which is also a complex reflection group by a result of Steinberg. In~\cite{YH2}, the second author defined a generalized Hecke algebra $\widetilde{H}_0$ attached to the \emph{normalizer} $N_0 = N_W(W_0)$, which is a natural extension of the Hecke algebra $H_0$ of $W_0$ by the group algebra of $\overline{N_0} = N_0/W_0 = N_W(W_0)/W_0$. This algebra turns out to be particularly useful for understanding (up to Morita equivalence) the `braid subalgebra' of the Yokonuma-Hecke algebras introduced in~\cite{YH1}.

It was proved in~\cite{YH2} that $\widetilde{H}_0$ is a free module over its ring of definition,
with a direct sum decomposition $\widetilde{H}_0 \simeq \bigoplus_{g \in \overline{N_0}} [g H_0]$
as a free $H_0$-module of rank $|\overline{N_0}|$. Since
it has been proven by Muraleedaran and Taylor 
in~\cite{TAYLORNORM} that the extension 
\begin{eqnarray}
\label{eq:sesW}
1 \to W_0 \to N_0 \to \overline{N_0} \to 1
\end{eqnarray}
is always split, it is expected that $\widetilde{H}_0$ is isomorphic to a semidirect (or crossed) product $\overline{N_0} \ltimes H_0$.

In~\cite{NORMARTIN}, Henderson and the authors positively answered this question when
$W$ is a \emph{real} reflection group (and actually in this case the reflection subgroup $W_0$ does not even need
to be parabolic), regardless of the ring of definition $\KK$ and the defining parameters of $\widetilde{H}_0$. In the present paper, we explore the general case, for which conditions need to be added. We
assume that $\KK$ is a domain and denote by $\KK^{\times}$ its group of invertible elements.

Our first main result is the following Theorem (see Theorem~\ref{theo:generic} below for a more precise statement):

\begin{theorem} Let $W_0$ be a parabolic subgroup of $W$. If the defining parameters
of $H_0$ are generic, and $\KK$ is a sufficiently large field of characteristic $0$, then
$\widetilde{H}_0 \simeq \overline{N_0} \ltimes H_0$.
\end{theorem}

However, this does not apply in general to the non-generic case.
In this paper we find explicit, sufficient algebraic conditions to ensure such a semidirect product
decomposition, using the classification of irreducible complex reflection groups. Indeed, it is not difficult to
see that for this problem we can assume that $W$ is irreducible.

In the case of the general series $G(de,e,n)$ of complex reflection groups,
we will prove that these problems can be reduced to the case
of a parabolic subgroup of the form
$$W_0=G(de, e, n_0)\times\prod_{k=1}^{n} G(1,1,k)^{b_k}$$
(see Section~\ref{sect:deen} for precise definitions). The Hecke algebra of the group $G(1,1,k)$ is the Hecke algebra
of type $A_{k-1}$ associated to the symmetric group $\mathfrak{S}_k$ (considered as
a Coxeter group). Let $\Delta(k)$ denote the element of its standard
basis associated to the element of maximal length of $\mathfrak{S}_k$ -- which is the image of Garside's fundamental
element of the usual braid group on $k$ strands. Our main result for the general series
is the following one, proved in Sections~\ref{sec:gr1n} and \ref{sect:deen}:

\begin{theorem} \label{theo:introGdeen} Let $W = G(de,e,n)$ and $W_0=G(de, e, n_0)\times\prod_{k=1}^{n} G(1,1,k)^{b_k}$.
Then $\widetilde{H}_0 \simeq \overline{N_0} \ltimes H_0$
as soon as, whenever $b_k \neq 0$, 
\begin{itemize}
\item  there exists $T_k \in \KK[X]$
such that $T_k(\Delta(k))^{-de} = \Delta(k)^2$, where the equality holds inside the Iwahori-Hecke algebra of type $A_{k-1}$, and
\item if moreover $e \neq 1$ and $n_0 \geq 1$, there exists 
$T_{0,k} \in \KK[X]$ such that $T_{0,k}(\sigma)^{de} = \sigma^{kd}$
whenever $\sigma$ is a braided reflection associated to the hyperplane $z_1 = 0$.

\end{itemize}
In particular the second condition is void when $d= 1$, as $z_1 = 0$ is not a reflecting hyperplane in that case.

\end{theorem} 

In most cases, the above conditions on the existence of polynomials
have a natural translation in terms of the parameters of the Hecke algebra
of $W_0$ (see Lemmas~\ref{lem:polfield} and \ref{lem:typea} below).

\medskip

In exceptional types, we determine semidirect product decompositions for all parabolic subgroups of maximal rank,
as well as for some parabolic subgroups of rank $1$. In rank $3$, we do this for
all the groups except $G_{27}$ in the notation of Shephard and Todd. Since, in rank $3$,
non-trivial proper parabolic subgroups either have rank $1$ or are maximal, this
solves our problem for these groups (that is, for $G_{24}$, $G_{25}$ and $G_{26}$).
We also solve it for the rank $4$ group $G_{32}$, for which we also have
to consider parabolic subgroups of rank $2$. In particular, we get the
following general result (see Theorem~\ref{theo:maxipars} for more details), where $B_0$ is the braid group of $W_0$ and the element $z_{B_0}$ is defined in Section~\ref{sect:genres}.

\begin{theorem} \label{theointro:maxis} Let $W$ be an irreducible complex reflection group of exceptional type,
and $W_0$ a parabolic subgroup of maximal rank. Let $z_{B_0}$ be the canonical
positive central element of $B_0$. Except for two exceptions for ranks $3$ and $5$,
if there exists $T \in K[X]$ such that the equality $T(z_{B_0})^{|Z(W)|} = z_{B_0}^{-|Z(W_0)|}$
holds inside $H_0$, then $\widetilde{H}_0 \simeq \overline{N_0} \ltimes H_0$.

\end{theorem}

In the last Section~\ref{sect:remaining}, we explore the remaining exceptional cases. There, we explain
in particular why a systematic exploration failed for the largest cases, and we nevertheless
manage to solve the problem for some of them, including all the (Shephard) groups whose braid group
is an Artin group.

\medskip
\textbf{Acknowledgments}. The first author thanks Vincent Beck and Anthony Henderson
for useful discussions. He was funded by Australian Research Council grant DP170101579 at early stages of this work.
Both authors thank Gunter Malle for a careful reading and annotations.

\section{General results}
\label{sect:genres}

\subsection{The Hecke algebra of a normalizer}

Let $W< \GL_n(\C)$ be a finite complex
reflection group
and $\mathcal{A}$ be the collection of its reflecting hyperplanes.
For such a collection of hyperplanes, we shall use the set-theoretic notation $\bigcup \mathcal{A}$
for the union of its elements. Its complement
$X = \C^n \setminus \bigcup \mathcal{A}$ is acted upon by $W$, and the braid group of $W$ is defined by $B = \pi_1(X/W)$. We denote $\pi : B \onto W$ the natural projection. Its kernel $P =\pi_1(X)$ is called the pure braid group of $W$. Finally, $\mathcal{A}$ is in 1-1 correspondence with
the set of distinguished reflections, namely the reflections $s \in W$ whose non-trivial eigenvalue
is equal to $\zeta_m$ for $\zeta_m = \exp(2 \pi \ii/m) \in \C^{\times}$ and $m$ equal to the order of
the cyclic subgroup fixing $\Ker(s - 1)$.

The group $B$ contains an important central element, which we denote $z_B$. When $W$ is irreducible, its center $Z(W)$ is cyclic of some order $m$, generated by
$\zeta_m \Id$.
 In this setting, $z_B$ is the homotopy class inside $X/W$ of the path $t \mapsto \exp(2 \pi \ii t /m) .*$, where $* \in X$ is the chosen base-point. In the general case, the ambient space $\C^n$ admits a canonical direct sum decomposition
$\C^n \simeq \C^{n_1}\oplus \dots \oplus \C^{n_r}$ yielding a decomposition
$W \simeq W_1 \times \dots \times W_r$, where $W_i < \GL_{n_i}(\C)$
is an irreducible reflection group. Letting $m_i = |Z(W_i)|$,
then $z_B$ is the homotopy class inside $X/W$ of the path $t \mapsto (\exp(2 \pi \ii t /m_1) .*_1,\dots,\exp(2 \pi \ii t /m_r) .*_r)$,
where $* = (*_1,\dots,*_r) \in X \subset \C^{n_1}\oplus \dots \oplus \C^{n_r}$ is the chosen base-point.
Finally, $B$ also contains as remarkable elements the braided reflections associated to the reflections of $W$
(see e.g. \cite{BMR,BESSISKPI1} for their precise geometric definition).

We recall from~\cite{BMR} the construction of the Hecke algebra $H$ of $W$ over some commutative ring $\KK$. It is defined using parameters $u_{i,s} \in \KK^{\times}$ for $s$ running among the distinguished reflections of $W$, where $0 \leq i < o(s)$ and $u_{i,s} = u_{i,t}$ when $s,t$ belong to the same conjugacy class. Then $H$ is the quotient of $\KK B$
by the relations $\prod_{i} (\sigma - u_{i,s}) = 0$ for every braided reflection $\sigma$ associated to $s$ -- so that its most general definition ring is the ring of Laurent polynomials $\Z[u_{i,s}^{\pm 1}]$. Its basic structural property is the now proven
BMR freeness conjecture, as a combination of \cite{ARIKI,ARIKIKOIKE,BROUEMALLE,HECKECUBIQUE,CYCLO,MARINPFEIFFER,CHAVLI,G20G21,TSUCHIOKA}.

\begin{theorem} \label{theo:freeness}%
The algebra $H$ is a free $\KK$-module of rank $|W|$.
\end{theorem}

Some cyclic extensions of $H$, related to the normalizer in $\GL_n(\C)$, were
previousy considered in \cite{MALLESPLIT}. Here we consider the following setting.
Let $W_0< W < \GL_n(\C)$ be a reflection subgroup of $W$, that is a subgroup of $W$ generated by some of its
reflections, and denote $\mathcal{A}_0$ its hyperplane arrangement. Recall from \cite{YH2} that such a subgroup is called \emph{full} if, for every
reflection it contains, all the reflections of $W$ fixing the same hyperplane also belong to $W_0$. 
 We consider the normalizer $N_W(W_0)=N_0$ of $W_0$ inside $W$. By definition, the \emph{Hecke algebra $\widetilde{H}_0$ of $N_0$}
as defined in \cite{YH2} is a quotient of the group algebra $\KK \hat{B}_0$ of
$\hat{B}_0 = \pi^{-1}(N_0)= \pi_1(X/N_0)$, by two types of relations:

\begin{itemize}
\item The relations $\sigma^{m_L} = 1$, for every braided reflection $\sigma$ associated to a hyperplane $L \in \mathcal{A}\setminus \mathcal{A}_0$. Here $m_L$ is the order of the pointwise stabilizer of $L$ in $W$. Notice
that $\sigma^{m_L} \in P \subset \hat{B}_0$ even when $\sigma \not\in \hat{B}_0$.
\item The defining relations of the Hecke algebra $H_0$ of $W_0$ on the braided reflections with respect to hyperplanes in $\mathcal{A}_0$.
\end{itemize}

We have the following generalization of Theorem \ref{theo:freeness}, proven in
\cite{YH2}.

\begin{theorem} \label{theo:freenessnorm}%
The algebra $\widetilde{H}_0$ is a free $H_0$-module of
rank $|\overline{N_0}| = |N_0/W_0|$, with a natural direct sum decomposition 
$$
\widetilde{H}_0 = \bigoplus_{g \in \overline{N_0}} [ g H_0]
$$
with $[g H_0]$ a free right $H_0$-module of rank $1$.
As a consequence it is a free $\KK$-module of rank $|N_0|$.
\end{theorem}

An equivalent definition of $\widetilde{H}_0$ can be given as follows. We introduce the normal subgroup $\QQ0$ of $\hat{B}_0$ generated by all the
$\sigma^{m_L}$, for $\sigma$ a braided reflection around some hyperplane $L \in \mathcal{A}\setminus \mathcal{A}_0$. Let  
$\widetilde{B}_0 = \hat{B}_0/\QQ0$. We define $\widetilde{H}_0$ as the quotient of $\KK \widetilde{B}_0$ by
the Hecke relations of $W_0$, which makes sense as all the braided reflections of $B$ with respect to a hyperplane in $\mathcal{A}_0$ belong to $\hat{B}_0$. These elements of the
form $\sigma^{m_L}$ are exactly the \textit{meridians} around $L$, in the terminology of \cite{BESSISKPI1} (also called \textit{generators-of-the-monodromy} in \cite{BMR}). We set $X_0 = \C^n \setminus \bigcup \mathcal{A}_0$. Letting $B_0 = \pi_1(X_0/W_0)$ denote the braid group of $W_0$, we have
a short exact sequence of groups (see \cite[Section 2.2]{YH2})
\begin{eqnarray}\label{ses}
1 \to B_0 \to \widetilde{B}_0 \to \overline{N_0}
\to 1
\end{eqnarray}
and the direct sum decomposition of Theorem \ref{theo:freenessnorm} is such that $[g H_0] \subset \widetilde{H}_0$
is equal to $b H_0 = H_0 b$ for $b \in \widetilde{B}_0$ having $g \in \overline{N_0}$ for image. We have $b H_0 b^{-1} = H_0$ for $b \in \widetilde{B}_0$.

\begin{lemma} \label{lem:semidirect} Assume that there exists a group homomorphism $\psi : \overline{N_0} \to \widetilde{H}_0^{\times}$
such that, for every $g \in \overline{N_0}$, there exists $b \in \widetilde{B}_0$ with the following property
\begin{itemize}
\item $b$ maps to $g$ under $\widetilde{B}_0 \onto \overline{N_0}$
\item $\psi(g)$ belongs to the image $b H_0 \subset \widetilde{H}_0$ of $b (K B_0)$ under $K \widetilde{B}_0 \to \widetilde{H}_0$.
\end{itemize} 
Then $\widetilde{H}_0 \simeq \overline{N_0} \ltimes H_0$. In particular, if the short exact sequence (\ref{ses}) splits,
then $\widetilde{H}_0$ is a
semidirect product  $\overline{N_0}\ltimes H_0$.
\end{lemma}

\begin{proof}
Because of $\psi(g) \in \widetilde{H}_0^{\times}$ and the stated conditions, we have $\psi(g) H_0  = b H_0 = [g H_0]$.
Writing $\psi(g) = b m$ for $m \in H_0$, since $b$ and $\psi(g)$ are invertible we get that $m$ is also invertible, and
$\psi(g) H_0 \psi(g)^{-1} = b m H_0 m^{-1} b^{-1} = b  H_0  b^{-1} = H_0$.
It follows that there is an algebra morphism $\overline{N_0} \ltimes H_0 \to \widetilde{H}_0$
mapping $g \otimes x$ for $g \in \overline{N_0}$, $x \in H_0$,
to $\psi(g) x$. Since it maps each $g \otimes H_0$ to $[g H_0]$ bijectively,
this is an isomorphism $\overline{N_0} \ltimes H_0 \simeq \widetilde{H}_0$.
\end{proof}

The following has been proven
in \cite[Theorems 3.15 and 3.19]{NORMARTIN}:

\begin{theorem}
\label{theo:liftreal}
If $W$ is a finite real reflection group and $W_0$ is an arbitrary
reflection subgroup of $W$, then the short exact sequence~(\ref{ses}) splits and 
$\widetilde{H}_0\simeq \overline{N_0} \ltimes H_0$.
\end{theorem}

In~\cite[Proposition 5.1]{NORMARTIN}, it is moreover shown that the short exact sequence~(\ref{ses})
 also splits in the case where $W$ is the complex reflection group $G(r,1,n)$ and $W_0$ is a standard parabolic subgroup of type $G(d,1,k)$, $k\leq n$. However, this conclusion cannot be expected for arbitrary finite complex reflection groups. Indeed, the splitting of the short exact sequence~(\ref{ses}) implies the splitting of the short exact sequence~(\ref{eq:sesW}), 
but there are pairs $(W,W_0)$ where $W$ is a reflection group and $W_0$ a reflection subgroup of $W$ such that the short exact sequence above does \textit{not} split (see \cite[Section 6]{NORMARTIN}). Nevertheless, we will show in Subsection~\ref{sec:generic} below that,
generically in characteristic $0$, this is the only obstruction for a semidirect product decomposition of the Hecke algebra $\widetilde{H}_0$ of $N_0$.

For later use, we prove the following result:

\begin{lemma}\label{lem:conjug}
Let $W_1,W_2$ be two reflection subgroups of $W$ which are conjugate, let $G_i =N_W(W_i)$,
and $\hat{B}_i, \widetilde{B}_i, B_i,\widetilde{H}_i$ be the groups and algebras $\hat{B}_0, \widetilde{B}_0, B_0,
\widetilde{H}_0$ attached to $W_0 = W_i$ as above, $i=1,2$. 

Then there is a group isomorphism $\widetilde{B}_1 \to \widetilde{B}_2$ mapping $B_1$ to $B_2$ 
inducing an algebra isomorphism $\widetilde{H}_1 \to \widetilde{H}_2$ such that $H_1$ is mapped to $H_2$. 
\end{lemma}

\begin{proof}
Let $w \in W$ such that $W_2 = w W_1 w^{-1}$ and $b \in \pi^{-1}(\{ w \}) \in B$. Setting $G_i = N_W(W_i)$, we have $G_2 = w G_1 w^{-1}$,
hence $\hat{B}_{2} = \pi^{-1}(G_2) = \pi^{-1}(wG_2w^{-1}) = b \pi^{-1}(G_1) b^{-1} = b \hat{B}_{1} b^{-1}$.
Let $Q_i = \Ker(\hat{B}_{i} \onto B_{i})$. By definition, $Q_i$ is generated as a group by the set of all the meridians around the reflecting hyperplanes of $W$ which are not reflecting hyperplanes for $W_i$.
 Now, $x \mapsto b x b^{-1}$ realizes a bijection between the meridians,
mapping the generating ones for $Q_1$ to the generating ones for $Q_2$,
hence $b Q_1 b^{-1} = Q_2$. It follows that $x \mapsto b x b^{-1}$ restricts to an isomorphism $\hat{B}_{1} \to \hat{B}_{2}$ which maps $Q_1$ to $Q_2$, therefore induces an isomorphism $\widetilde{B}_{1} \to \widetilde{B}_2$. Since $Q_1,Q_2$ are subgroups of the pure braid group $P$, this isomorphism fits into a commutative diagram of the form
$$
\xymatrix{
\widetilde{B}_1 \ar[r]\ar[d]_{\pi} & \widetilde{B}_2 \ar[d]_{\pi} \\
G_1 \ar[r] & G_2  
}
$$
where $G_1 \to G_2$ is the map $x\mapsto wxw^{-1}$. Since $W_2 = w W_1 w^{-1}$ this implies that it maps $B_1 = \Ker(\widetilde{B}_1 \onto N_W(W_1)/W_1)$ to $B_2$. Finally, since $x \mapsto b x b^{-1}$ maps braided reflections around reflecting hyperplanes of $W_1$
to braided reflections around reflecting hyperplanes of $W_2$, the defining ideal of $\widetilde{H}_1$ inside $\KK \widetilde{B}_1$ is mapped to the defining ideal of $\widetilde{H}_2$ inside $\KK \widetilde{B}_2$, and this induces an isomorphism $\widetilde{H}_1 \to \widetilde{H}_2$, which maps the image of $\KK B_1$ inside $\widetilde{H}_2$ to the image of $\KK B_2$ inside $\widetilde{H}_2$, namely $H_1$ to $H_2$, and this proves the claim.

\end{proof}

\subsection{The generic Hecke algebra of the normalizer}\label{sec:generic}

Let $W < \GL_n(\C)$ be a complex reflection group, $\mathcal{R}$ the set of its (pseudo-)reflections, and let $\mathcal{R}^* \subset \mathcal{R}$ be the collection of the distinguished ones. Let $W_0 < W$ be a full
reflection subgroup of $W$, and $\mathcal{R}_0 \subset \mathcal{R}$, $\mathcal{R}_0^* \subset \mathcal{R}^*$ its
collection of (distinguished) reflections.

\subsubsection{An isomorphism \`a la Cherednik}

There is a natural bijection $\mathcal{R}^* \to \mathcal{A}$ given by $s \mapsto \Ker(s-1)$. We denote $L \mapsto s_L$ its inverse. For any choice of elements $\varphi_L \in \C W_L$ with $W_L =\langle s_L \rangle$, $L \in \mathcal{A}$, with the condition that $\varphi_{w(L)} = w\varphi_Lw^{-1}$ for every $w \in W$, it is well-known (see for instance~\cite{BMR}) that the $1$-form
$\sum_{L \in \mathcal{A}} h \varphi_L \om_L \in \Omega^1(X) \otimes \C[[h]] W$, with $\om_L$ the logarithmic 1-form over $X$ associated to $L$ (that is, $\dd \alpha_L/\alpha_L$ for $\alpha_L$ any linear form defining $L$), is integrable and $W$-equivariant and provides an algebra isomorphism $H \to K W$, as a
consequence of Theorem~\ref{theo:freeness}, where $\KK = \C((h))$ is the field of Laurent series, and the $u_{s_L,k} \in K^{\times}$ depend on $\varphi_L$.

More precisely, we have $\C \langle s_L \rangle = \prod_{k=0}^{m_H-1} \Ker(s_L - \zeta_L^{k}) \simeq \C^{m_L}$,
where we identified $s_L$ with its image under the multiplication operator map $\C \langle s_L \rangle \into \End(\C \langle s_L \rangle)$, and $\zeta_L=\exp(2 \ii \pi/m_L)$. We denote $\eps_{L,k}$ the primitive
idempotent associated with $\Ker(s_L - \zeta_L^{k})$. Then letting $\varphi_L = \sum_{k=0}^{m_H-1}  \la_{L,k} \eps_{L,k} \in \C \langle s_H \rangle$ with scalars $\la_{L,k}$ chosen so
that $\varphi_{w(L)} = \varphi_L$ for all $L \in \mathcal{A},w \in W$, we get a morphism $ H \to K W$ with $u_{s_L,k} = \exp(2 \ii  \pi \la_{L,k} h/m_L)$,
which we call the parameters of $L$ associated to the collection of $\varphi_L, L \in \mathcal{A}$.
This morphism is an isomorphism as soon as the $u_{s_L,k}$ ($L \in \mathcal{A}, 0 \leq k \leq m_H-1$) are algebraically independent over $\C$,
which holds as soon as the
$\la_{L,k}$ are linearly independent over $\Q$. We call such a choice of parameters a \emph{generic choice} for $\varphi_L$.

We denote $\widetilde{H}_0'$ and $H'_0$ the Hecke algebras of $N_0$ and $W_0$
defined over $\C[[h]] \subset \C((h)) = K$. We prove the following.
\begin{proposition}\label{prop:cherednik}
For any choice of elements $\varphi_L \in \C W_L$, $L \in \mathcal{A}_0$, such that $\varphi_{g(L)} = \varphi_L$ for all $g \in N_0$, the 1-form 
$$
\om_0 = \sum_{L \in \mathcal{A}_0} \varphi_L \om_L \in \Omega^1(X) \otimes \C W_0
$$
is integrable over $X$, and $N_0$-equivariant. The monodromy of $h \om_0$ over $X/N_0$ provides an algebra homomorphism $\C[[h]] \widetilde{B}_0 \to \C[[h]] N_0$ mapping $\C[[h]] B_0$ to $\C[[h]] W_0$. The latter morphism
factorizes through $ \widetilde{H}_0'$ and induces a $\C((h))$-algebra morphism $\widetilde{H}_0 \to \C((h)) N_0$ mapping $H_0$ to $\KK W_0$ for the parameters of $H_0$ associated to the chosen values of $\varphi_L, L \in \mathcal{A}_0$. If this choice is generic in the above sense, then
 these algebra morphisms
  $ \widetilde{H}_0 \to K N_0$ and $ H_0 \to \KK W_0$
  are isomorphisms.
\end{proposition}

\begin{proof}
The 1-form $\om_0$ is the restriction to $X$ of the usual 1-form on $\C^n \setminus \bigcup \mathcal{A}_0$ attached to $W_0$, therefore it is integrable as well. The $N_0$-equivariance is clear. Thus, by e.g. Chen's iterated integrals as in \cite{CHEN}, one gets that the monodromy of $h \om_0$ 
provides a morphism $\C[[h]] \widehat{B}_0 = \C[[h]]\pi_1(X/N_0)
\to \C[[h]] N_0$ which restricts to the usual monodromy
morphism $\C[[h]] B_0 \to \C[[h]] W_0 \subset \C[[h]] N_0$. Since the map $\C \widehat{B}_0 \to \C N_0$ induced by $\pi : \widehat{B}_0 \to N_0$ is
surjective and coincides with reduction modulo $h$ of
$\C[[h]] \widehat{B}_0 \to \C[[h]] N_0$, by Nakayama's lemma one gets that the latter algebra homomorphism is surjective. Since the monodromy of $\om_0$ along meridians around $L \not\in \mathcal{A}_0$ is trivial, this morphism induces
an algebra morphism $\Phi : \C[[h]] \widetilde{B}_0 \to \C[[h]] N_0$ which is still surjective, and still extends $\C[[h]] B_0 \to \C[[h]] W_0$.

The latter is
known to factorize through $H_0'$, hence we
get that $\Phi$ induces an algebra morphism $\C[[h]] \widetilde{B}_0 \to \C[[h]] N_0$ which is still surjective,
and still extends $\C[[h]] B_0$ to $\C[[h]] W_0$. 

Since the latter
is known to factorize through $\widetilde{H}_0'$, we get that $\Phi$ induces a surjective algebra morphism
$\widetilde{H}_0 \to K N_0$. As a consequence
of Theorem~\ref{theo:freeness}, we have
equality of dimensions, therefore this provides an
isomorphism $ \widetilde{H}_0 \to K N_0$ mapping
$ H_0$ to $K W_0$.

\end{proof}

\subsubsection{Consequences in the generic case}

Let $\kk$ be
the ring of Laurent polynomials $\Z[u_{s,i}^{\pm 1}]$, where $s \in \mathcal{R}^*$ and $i \in \{ 0,\dots,o(s)-1\}$, with the
convention that $u_{s,i} = u_{wsw^{-1},i}$ for all $w \in N_0$. This is the most general ring
over which $\widetilde{H}_0$ is defined. 
In this section we consider the generic
case, that is, the case where $\KK$ is a field containing $\kk$. In particular $\KK$ has characteristic $0$.

We now assume that $W_0$ admits a complement inside $N_0$,
that is, we assume that $N_0 = U_0 \ltimes W_0$ for some $U_0 < N_0$.
We also also assume that the parameters of $H_0$ are generic in characteristic $0$,
that is, that $\KK$ is a field containing the generic ring $\kk = \Z[u_{s,i}^{\pm}]$.

For a generic choice of the parameters $\varphi_L$ of the previous section, we know by
Proposition~\ref{prop:cherednik} we know that there exists an algebra isomorphism $\Phi :  \widetilde{H}_0 \to \KK N_0$ mapping $ H_0$ to $\KK W_0$, for $\KK = \C((h))$
containing $\kk$ as a subfield. This implies that the same statement holds for every field extension $\KK \supset \kk$ which is large enough, for instance for an algebraic closure of $\kk$.

For such a field $\KK$ containing $\kk$, setting $M: = \Phi^{-1}(U_0)$, we get as a consequence that 
$ \widetilde{H}_0 = \bigoplus_{g \in M} g H_0$ and
$\widetilde{H}_0 = M \ltimes  H_0 \simeq \overline{N_0} \ltimes  H_0$. When $W_0$ is a parabolic subgroup of $W$, the existence of such a complement $U_0$ is proven in all cases in \cite{TAYLORNORM}, therefore a consequence of the above argument is the following:

\begin{theorem}\label{theo:generic}
Let $W_0$ be a parabolic subgroup of $W$. Then,
for $\KK$ a sufficiently large field containing $\kk = \Z[u_{s,i}^{\pm}]$, 
we have $\widetilde{H}_0 \simeq \overline{N_0} \ltimes H_0$.

\end{theorem}

\subsection{Genericity conditions}

Thanks to Theorem~\ref{theo:generic}, the question is therefore to determine algebraic criteria on the domain $\KK$ and on the parameters $u_{s,i} \in \KK$ so that the extension $\widetilde{H}_0$ is a semidirect product.

To this end, we will often need the following result of commutative algebra.

\begin{lemma} \label{lem:polfield} Let $\KK$ be an arbitrary domain,
let $a,b \in \Z\setminus\{ 0 \}$, and let $\chi(X) \in \KK[X]$ be a monic and split polynomial with
roots $v_1,\dots,v_r \in \KK^{\times}$ such that $i \neq j \Rightarrow v_i - v_j \in \KK^{\times}$.

We call such a split polynomial \emph{square-free}.
Then the following are equivalent: 
\begin{itemize}
\item There exists $T \in \KK[X]$ such that $T(X)^a \equiv X^b
\mod \chi(X)$ inside $\KK[X,X^{-1}]$,
\item The domain $\KK$ contains some $a$-th root of each 
$v_i^b, 1 \leq i \leq r$.
\end{itemize}

\end{lemma}

\begin{proof}
By assumption, we can write
$\chi(X) = (X-v_1)\dots (X-v_r)$. Because of the conditions $v_i \in \KK^{\times}$ and $v_i - v_j \in \KK^{\times}$,
applying the Chinese
Remainder Theorem we have
$$
\KK[X,X^{-1}]/(\chi) \simeq \KK[X]/(\chi) \simeq \prod_i \KK[X]/(X-v_i) \simeq \KK^r.
$$

Then the desired property of $T$ is equivalent to the equations $T(v_i)^a = v_i^b$. There is no
solution to such an equation if $v_i^{b}$ has no $a$-th root in $\KK$. If it has such an $a$-th root $v_i^{b/a}$,
 we are looking for a polynomial $T \in \KK[X]$
such that $T(v_i) = v_i^{b/a}$ for $1 \leq i \leq r$.
This is then a linear equation in
the coefficients of $T$, whose determinant is the Vandermonde determinant attached to the $v_i$'s.
This determinant is invertible in $\KK$ because of our assumptions, and this proves the claim.

\end{proof}

One specific element whose minimal polynomial will play a major role
is the following one.
\begin{lemma}\label{lem:typea}
The image of Garside's fundamental element in the braid group on $n+1$ strands
inside
the Hecke algebra of type $A_{n}$ is annihilated by the polynomial
\begin{equation}
\label{eq:polmintypeA}
\prod_{\stackrel{i,j \geq 0}{i+j = n(n+1)}} (X^2-(-1)^i u_0^i u_1^j)
\end{equation}
with $u_0 = u_{s,0}$ and $u_1 = u_{s,1}$ for any $s\in \mathcal{R}$.
\end{lemma}
\begin{proof}
The image in the statement is equal to the element $T_{w_0}$ of the standard basis
of the Hecke algebra, associated to the longest element $w_0$ of
the symmetric group $\mathfrak{S}_{n+1}$, which has length $n(n+1)/2$. Setting $q = -u_1u_0^{-1}$
and renormalizing each braided reflection $\sigma$ as $-u_0^{-1}\sigma$,
we get the equivalent formulation that, inside the Hecke algebra
defined over $\Z[q^{\pm 1}]$ by the equation $(\sigma-q)(\sigma +1)$ as in the conventions of \cite{GECKPFEIFFER},
the element $T_{w_0}$ is annihilated by the
polynomial $\prod_{0 \leq i \leq n(n+1)} (X^2-q^i)$. This Hecke algebra is semisimple over the algebraic closure
$\overline{\Q(q)}$ of the field
of fractions of $\Z[q^{\pm 1}]$, and the central element $T_{w_0}^2$ acts by
the scalar $z_{\la} \in \overline{\Q(q)}$ on the irreducible representation 
attached to the partition $\la \vdash n+1$. 

Therefore $T_{w_0}$ is annihilated
by the polynomial $\prod_{a \in A} (X^2 - a)$, where
$A$ denotes the set of
all values $v_{\la}, \la \vdash n+1$,
inside the Hecke algebra
over $\overline{\Q(q)}$. This happens already inside
the original Hecke algebra defined over $\Z[q^{\pm 1}]$, because it is a free module
over it. We thus only need to prove that each of the $z_{\la}$ is of the form $q^i$, $0 \leq i \leq n(n+1)$.

By a result of Springer (see \cite{GECKPFEIFFER} 9.2.2), we have $z_{\la} = q^{v_{\la}\frac{n(n+1)}{2}}$
with $v_{\la}\frac{n(n+1)}{2} \in \Z$, where
$v_{\la} = \gamma(\la)/\dim(\la)$,  with $\dim(\la)$ 
of the irreducible representation of the symmetric group $\mathfrak{S}_n$ attached to $\la$, and $\gamma(\la)$ is the trace of $1 + (1\ 2)$. We need to prove that $v_{\la}$ always belongs to the real interval $[0,2]$.

Notice that $v_{[1^{n+1}]} = 0$
and $v_{[n+1]} = 2$, and in particular the statement on the $z_{\la}'s$ is true for $n+1 = 2$.
We proceed to prove it by induction on $n$. For $\mu \vdash n$, let us denote $[\la : \mu]$ the multiplicity
of (the representation of $\mathfrak{S}_n$ attached to) $\mu$ in the restriction
of $\la$. Then, all sums being understood over all possible $\mu$'s, we have $\dim(\la) = \sum [\la : \mu] \dim(\mu)$
and $\gamma(\la) = \sum [\la : \mu] \gamma(\mu)$. As a consequence,
$$
v_{\la} = \gamma(\la)/\dim(\la) 
= 
\frac{1}{\dim(\la)}\sum [\la : \mu] \gamma(\mu)
= 
\frac{\sum [\la : \mu]\dim(\mu) v_{\mu}}{\sum [\la : \mu] \dim(\mu)}
$$
hence $v_{\la}$ belongs to the minimal interval containing all the $v_{\mu},\mu \vdash n$. By the
induction assumption this interval is included in (and even equal to) $[0,2]$,
and this concludes the induction step and the proof.

\end{proof}

Actually, for any given $n$, the proof provides a more specific polynomial,
as every $z_{\la}$ is explicitely computable. For the small
values of $n$, the index $i$ in the formula (\ref{eq:polmintypeA}) belongs to
the sets given in Table~\ref{tab:tabulationDeltaAn}.
\begin{table}
$$\begin{array}{|c|c|}
\hline
n+1 & i \\
\hline
\hline
2 & \{ 2, 0 \} \\
\hline
3 & \{ 6, 3,0 \} \\
\hline
4 & \{ 12, 8,6,4,0 \} \\
\hline
5 & \{ 20,15,12,10,8,5, 0 \} \\
\hline
6 & \{ 30,24,20,18,15,12,10,6,0 \} \\
\hline
\end{array}
$$
\caption{Eigenvalues of Garside's element in type $A$}
\label{tab:tabulationDeltaAn}
\end{table}

In particular, whether
the polynomial (\ref{eq:polmintypeA}) splits or not
highly depends on $n$, as it is related to the appearance of
odd integers
among the possible values of $i$.

\medskip

\section{Groups of type $G(r,1,n)$}\label{sec:gr1n}

In this section, we prove Theorem~\ref{theo:introGdeen} in the case where $W$ is the complex reflection group $G(r,1,n)$, that is, we find explicit conditions to ensure that we have an isomorphism $\widetilde{H}_0 \simeq \overline{N_0} \ltimes H_0$
in this case. 

In~\cite[Proposition 5.1]{NORMARTIN}, Henderson and the authors showed that if $W_0$ is a parabolic subgroup of type $G(r,1,k)$ of $W$ ($k\leq n$), then the group-theoretic
short exact sequence~\eqref{ses} splits, and we even have a direct product decomposition $\widetilde{B}_0= B_0\times \overline{N_0}$ in that case,
yielding a direct tensor product decomposition $\widetilde{H}_0 = H_0 \otimes \KK \overline{N_0}$ (as a $K$-algebra).

In particular, $\widetilde{H}_0 \simeq \overline{N_0} \ltimes H_0$
without any condition on the parameters. However, this does \textit{not} hold in general, as shown in~\cite[Example 6.6]{NORMARTIN}.

We begin by a few observations on normalizers of parabolic subgroups and the existence of complements (proven in~\cite{TAYLORNORM}).

\subsection{Special elements and parabolic subgroups}

Let $n\geq 1, r\geq 2$. Recall that $W=G(r,1,n)$ has a Coxeter-like presentation with $n$ generators $s_0, s_1, \dots, s_{n-1}$, with relations given by the type $B_{n}$ braid relations (with $s_0 s_1 s_0 s_1=s_1 s_0 s_1 s_0$), together with the relations $s_0^r=1$, $s_i^2=1$ for all $1\leq i \leq n-1$. This is exactly the Coxeter presentation of the Weyl group of type $B_n$, except that $s_0$ has order $r$ (for $r=2$ we recover the Weyl group of type $B_n$). 

It is shown in~\cite[Theorem 3.9]{TAYLORREFL} that every parabolic subgroup of $W$ is conjugate to a parabolic subgroup (which we will call \textit{standard}) generated by a subset $S_0$ of the set $S:=\{ s_0, s_1, \dots, s_{n-1}\}$ of Coxeter-like generators of $W$. By Lemma~\ref{lem:conjug}, it is enough to deal with the standard parabolic subgroups in order to prove $\widetilde{H}_0 \simeq \overline{N_0} \ltimes H_0$ -- which, in certain cases, may be achieved thanks to Lemma~\ref{lem:semidirect} by showing the existence of a splitting of the short exact sequence~(\ref{ses}).

Recall that $W$ is
the group of $(n\times n)$-monomial matrices whose entries are $r$-th roots of unity. In this description, $s_0$ is the diagonal matrix having $\zeta = \exp(2\pi \ii/r)$
as first entry and $1$ everywhere else, with $s_1, s_2, \dots, s_{n-1}$ acting on the basis vectors by permuting them in the obvious way. The action of $W$ on $\C^n$ is irreducible.

Let $S'=\{s_1, s_2, \dots, s_{n-1}\}$, $\mathfrak{S}_n=\langle S'\rangle$. We have $G(r,1,n)= (\mathbb{Z} / r \mathbb{Z})^n \rtimes \mathfrak{S}_n$,
where the generator $t_i$ in the $i$-th factor of $(\mathbb{Z} / r \mathbb{Z})^n$ is given by the diagonal matrix having $\zeta=\det(s_0)$ as $i$-th entry and $1$ everywhere else. It is a reflection, and can be expressed as a product of elements of $S$ as $s_0$ if $i=1$ and $s_{i-1} \cdots s_1 s_0 s_1 \cdots s_{i-1}$ if $i\geq 2$. Note that $(\mathbb{Z} / r \mathbb{Z})^n$ can also be defined as the subgroup generated by the $\mathfrak{S}_n$-conjugates of those elements which lie in $S\setminus S'=\{s_0\}$, generalizing in this specific case the semi-direct product decomposition of Coxeter groups given in~\cite{BONNAFEDYER}. 

The braid group $B=B(r,1,n)$ of $G(r,1,n)$ is isomorphic to the Artin group of type $B_n$. We denote its standard Artin generators by $\sigma_0, \sigma_1, \dots, \sigma_{n-1}$.

In~\cite{BREMKEMALLE}, it is shown that every element $w\in W$ can be written uniquely in the form $$t_{0,k_0} t_{1, k_1}\cdots t_{n-1, k_{n-1}}p(w),$$
where $p(w)\in\mathfrak{S}_n$ and $0\leq k_i \leq r-1$. Here $t_{i,k}$ is defined as $s_i s_{i-1} \cdots s_1 s_0^{k}$ if $k\neq 0$, and as $1$ if $k=0$. 

Let $W(r)$ denote the set that consists of those $w\in W$ such that in the above normal form, we have $k_i\in\{0,1\}$ for all $i=0,\dots, n-1$. Note that it is exactly the set of monomial matrices in $\{1, \zeta\}$.

Replacing every $t_{i,k}$ by $s_i s_{i-1} \cdots s_0^{k}$ if $k\neq 0$ and by $1$ if $k=0$ in the above normal form and $p(w)$ by a reduced expression of it in $\mathfrak{S}_n$, by~\cite[Lemma 1.5]{BREMKEMALLE} one obtains a reduced expression in the generating set $S$ of $W$. Moreover, the same lemma shows that, if we restrict to $W(r)$, then taking any reduced expression of $w\in W(r)$ in $S$ and replacing every $s_i$ by $\sigma_i$ yields a well-defined element $\mathbf{w}$ of $B$. Indeed, here $B$ is an Artin group of type $B_n$, and \cite[Lemma 1.5]{BREMKEMALLE} restricted to $W(r)$ shows that any two reduced expressions of $w\in W(r)$ can be related by the defining relations of $B$. Let $\mathcal{D}:=\{\mathbf{w}~|~w\in W(r)\}$. Note that $\mathcal{D}$ is independent of $r$. 

In fact, the set $\mathcal{D}$ is the set of (left or right) divisors of the Garside element $(\sigma_0\sigma_1 \dots \sigma_{n-1})^n$ of the braid monoid $B^{+}\subseteq B$, which is a Garside monoid. In other words, the morphism of monoids $B^{+}\longrightarrow W$ induces by restriction an injective set-theoretic map $\mathcal{D}\rightarrow W$, with image equal to $W(r)$. 
Note that for $r=2$ we have $W=W(r)$ and $\mathcal{D}$ is just the set of positive lifts of elements of $W$, which is a Coxeter group of type $B_n$ in that case; for $r\neq 2$ the subset $W(r)$ is \textit{not} a subgroup of $W$. 

Denoting by $\ell$ the length function on $W$ with respect to $S$, what we just recalled implies the following important property, which will be used repeatedly below: 

\begin{lemma}\label{lift_rex}
Let $u,v\in W(r)$ such that $w:=uv$ lies in $W(r)$ and $\ell(u)+\ell(v)=\ell(w)$. Then $\mathbf{u} \mathbf{v}=\mathbf{w}$. 
\end{lemma}

\subsection{Description of the complements}

Howlett~\cite{HOWLETT} found a description of complements of parabolic subgroups of Coxeter groups inside their normalizers; it is not hard to see that his construction does not generalize to the $G(r,1,n)$ case, even though these groups are close to being Coxeter groups. In this section, we therefore recall a different way of finding complements of parabolic subgroups of $W$ inside their normalizers; this construction is due to Taylor and Muraleedaran~\cite{TAYLORNORM} and, up to slight variations, will be relevant to construct suitable lifts of these complements inside $\widetilde{H}_0$.

By the above remarks, we can assume that $S_0\subseteq S$ has at least one irreducible component which generates a Coxeter group of type $A$, since we have already a group-theoretic splitting otherwise, according to~\cite[Proposition 5.1]{NORMARTIN}.  

We first explain how to define a subgroup $U_0$ which is complementary to $W_0=\langle S_0 \rangle$ inside $N_0$ by recalling results from~\cite{TAYLORNORM}. We introduce some technical notation, illustrated in Example~\ref{ex_14} below. As in~\cite{TAYLORNORM}, we associate a partition of $n$ to $W_0$ by writing $W_0$ as a direct product of irreducible standard factors, also counting the trivial group, i.e., we write $$W_0=G(d, 1, n_0)\times\prod_{k=1}^{n} G(1,1,k)^{b_k}$$ in such a way that $1^{b_1} 2^{b_2} \cdots n^{b_{n}}$ is a partition of $n-n_0$. Let $b_{i_1}, b_{i_2}, \cdots, b_{i_\ell}$ be the sequence of those $b_k$ such that $b_k\neq 0$, where $i_1< i_2< \dots <i_\ell$. If $b_k\neq 0$, we write $R_k:=G(1,1,k)^{b_k}\cong \mathfrak{S}_k^{b_k}$. Let $d_{i_1}=n_0+1$, $d_{i_{j+1}}=d_{i_j}+b_{i_j} i_j$ for $1\leq j \leq \ell-1$. Up to replacing $W_0$ by some $\mathfrak{S}_n$-conjugate, if $k=i_j$, then we can assume that the $b_{k}$ irreducible factors of Coxeter type $A_{k-1}$ have generating sets given by $$J_k^1:=\{s_{d_k}, s_{d_k+1}, \dots, s_{d_k+k-2}\},$$
$$J_k^2:=\{s_{d_k+k}, \dots, s_{d_k+2k-2}\},$$
$$\dots,$$
$$J_k^{b_k}:=\{s_{d_k+(b_{k}-1) k}, \dots, s_{\underbrace{d_k + b_k k}_{=d_{i_{j+1}}}-2}\}.$$
We set $J_k=J_k^1 \cup J_k^2 \cup\cdots\cup J_k^{b_k}$ and we have $R_k=\langle J_k\rangle$ (note that $J_1^m=\emptyset$ for all $1\leq m \leq b_1$, as $J_1$ is the trivial group). With this choice of labeling, we get that $d_k$ indexes the leftmost node of $J_k$ in the Dynkin-like diagram associated to $W$ (see Example~\ref{ex_14} below), and the irreducible factors of type $A$ of $W_0$ are arranged from the left to the right, always separated by a single node, with their size increasing from left to right.

\begin{example}\label{ex_14}
Let $W=G(r,1,14)$ and $W_0=G(1,1,1)^3\times G(1,1,3)^2\times G(1,1,5)$. We have $n_0=0$ and $b_1=3$, $b_3=2$, $b_5=1$; $i_1=1, i_2=3, i_3=5$, and $d_1=1, d_3=4, d_5=10$. The nodes that are elements of $S_0$ are the red ones in the picture below. We have $J_3^1=\{s_4, s_5\}$, $J_3^2=\{s_7, s_8\}$, $J_5^1=J_5=\{s_{10}, s_{11}, s_{12}, s_{13}\}$.  

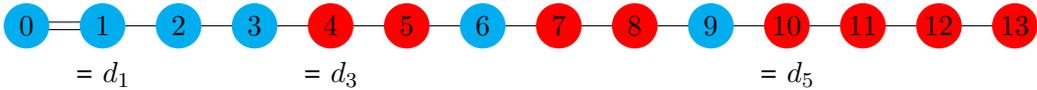
\begin{center}
\begin{tikzpicture}
\draw (0,0+.05) -- (1,0+.05);
\draw (0,0-.05) -- (1,0-.05);
\draw (1,0) -- (2,0);
\draw (2,0) -- (3,0);
\draw (3,0) -- (4,0);
\draw (4,0) -- (5,0);
\draw (5,0) -- (6,0);
\draw (6,0) -- (7,0);
\draw (7,0) -- (8,0);
\draw (8,0) -- (9,0);
\draw (9,0) -- (10,0);
\draw (10,0) -- (11,0);
\draw (11,0) -- (12,0);
\draw (12,0) -- (13,0);
\fill[color=cyan] (0,0) circle (.3);
\fill[color=cyan] (1,0) circle (.3);
\fill[color=cyan] (2,0) circle (.3);
\fill[color=cyan] (3,0) circle (.3);
\fill[color=red] (4,0) circle (.3);
\fill[color=red] (5,0) circle (.3);
\fill[color=cyan] (6,0) circle (.3);
\fill[color=red] (7,0) circle (.3);
\fill[color=red] (8,0) circle (.3);
\fill[color=cyan] (9,0) circle (.3);
\fill[color=red] (10,0) circle (.3);
\fill[color=red] (11,0) circle (.3);
\fill[color=red] (12,0) circle (.3);
\fill[color=red] (13,0) circle (.3);
\draw (0,0) node {$0$};
\draw (1,0) node {$1$};
\draw (1,-0.65) node {$=d_1$};
\draw (2,0) node {$2$};
\draw (3,0) node {$3$};
\draw (4,0) node {$4$};
\draw (4,-0.65) node {$=d_3$};
\draw (5,0) node {$5$};
\draw (6,0) node {$6$};
\draw (7,0) node {$7$};
\draw (8,0) node {$8$};
\draw (9,0) node {$9$};
\draw (10,0) node {$10$};
\draw (10,-0.65) node {$=d_5$};
\draw (11,0) node {$11$};
\draw (12,0) node {$12$};
\draw (13,0) node {$13$};
\end{tikzpicture}
\end{center}

\end{example}

We now assume that the conditions given in Theorem \ref{theo:introGdeen} are satisfied,
providing suitable polynomials $T_k,T_{0,k} \in \KK[X]$.
As shown in \cite{TAYLORNORM}, for every $k$ such that $b_k\neq 0$, there is a subgroup $N_k\subseteq N_0$, which normalizes $R_k$ (and centralizes every $R_j$ with $j\neq k$) and acts as the reflection group $G(r,1,b_k)$ on a suitable subspace of the natural module $V=\C^n$ (see~\cite[Theorem 3.12 and its proof]{TAYLORNORM}). Moreover, $U_0=\prod_{k, b_k\neq 0} N_k$ is a complement to $W_0$ inside $N_0$. We give explicit generators of the subgroups $N_k$, corresponding to the Coxeter-like generators of $G(r,1,b_k)$.

Let $e_1, e_2, \dots, e_n$ be the standard basis of $\C^n$, where $W < \GL_n(\C)$.
The direct product
decomposition of $W_0$ corresponds to a tensor product decomposition $\C^n = \C^{n_0} \oplus \bigoplus_{k=1}^{n}
\C^{k} \otimes \C^{b_{k}}$ of the space, where the element $a_p\otimes b_q$ of the
canonical basis of $\C^{k} \otimes \C^{b_{k}}$ is mapped to
$e_{d_k+(q-1)k+p-1}$,
 when $b_k \neq 0$.  Then $G(r,1,b_k)$ acts on each
$\C^{k} \otimes \C^{b_{k}}$ as $f_k = \un \otimes \rho$, where $\rho$ is the natural representation of $G(r,1,b_k)$ on $\C^{b_k}$ and $\un$ is the trivial representation
on $\C^k$. In particular, $s_k^{(0)}:=f_k(s_0)$ acts by multiplication by $\zeta$ on each $a_p \otimes b_1$
and by the identity on the other basis vectors, while $s_k^{(i)}:=f_k(s_i)$ acts
through $a_r \otimes b_p \mapsto a_r \otimes b_{s_i(p)}$.

In fact, we have $s_k^{(1)}, \dots, s_{k}^{(b_k-1)}\in \langle S'\rangle\subseteq \mathfrak{S}_n$, while $s_k^{(0)}=t_{d_k} t_{d_k+1} \cdots t_{d_k+k-1}\in (\mathbb{Z} / r \mathbb{Z})^n$. 

Hence the semidirect product decomposition of $N_k$ (obtained by viewing it as the reflection group $G(r,1,b_k)$) is compatible with the semidirect product decomposition $(\mathbb{Z} / r \mathbb{Z})^n \rtimes \mathfrak{S}_n$ of $W$, in the sense that it is the same as the intersection of the semidirect product decomposition of $W$ with $N_k$. In fact, as $s_k^{(1)}, \dots, s_{k}^{(b_k-1)}\in \mathfrak{S}_n \cap N_W(R_k)$, we have $s_k^{(1)}, \dots, s_{k}^{(b_k-1)}\in N_{\mathfrak{S}_n}(R_k)$, and they turn out to be among the generators of the complement to $R_k$ inside $\mathfrak{S}_n$ as described by Howlett~\cite{HOWLETT}: more precisely, for all $i=1, \dots, b_k-1$, we have \begin{equation}\label{eq_si} s_k^{(i)}=w_0(K_i) w_0(J_k^{i+1}\cup J_k^i),\end{equation} where $K_i=J_k^{i+1}\cup J_k^i\cup\{s_{d_k-1+ i k} \}$ is the connected closure of $J_k^i$ and $J_k^{i+1}$ (obtained by just adding the unique simple reflection between them in the Dynkin diagram) and $w_0(X)$ denotes the longest element in the Coxeter group of type $A$ generated by $X\subseteq S'$. 

Also note that $s_k^{(0)}, s_k^{(1)}, \dots, s_{k}^{(b_k-1)}$ all lie in $W(r)$.  

\begin{example}
Let $W=G(r,1,4)$, $S_0=\{s_1, s_3\}$. Then $W_0=G(1,1,2)^2$, hence $2$ is the only integer $k$ such that $b_k\neq 0$ and we have $S_0=J_2$, $b_2=2$. Hence $U_0=N_2\cong G(r,1,2)$, $E_2^1=\{e_1, e_2\}$, $E_2^2=\{e_3, e_4\}$, and the Coxeter-like generators $s_2^{(0)}$, $s_2^{(1)}$ of $N_2$ are given by the matrices $$s_2^{(0)}=
\begin{pmatrix}
  \begin{matrix}
  \zeta & 0 \\
  0 & \zeta
  \end{matrix}
  & \rvline & 0 \\
\hline
  0 & \rvline &
  \begin{matrix}
  1 & 0 \\
  0 & 1
  \end{matrix}
\end{pmatrix}
, ~s_2^{(1)}=
\begin{pmatrix}
0
  & \rvline &   \begin{matrix}
  1 & 0 \\
  0 & 1
  \end{matrix} \\
\hline
  \begin{matrix}
  1 & 0 \\
  0 & 1
  \end{matrix} & \rvline &
  0
\end{pmatrix}.$$
\end{example}

\subsection{Lifting the complement}

Let $\tau_1=\sigma_0\in B$ and for all $2\leq i \leq n$, let $$\tau_i:=\sigma_{i-1} \sigma_{i-2} \cdots \sigma_1 \sigma_0 \sigma_1 \cdots \sigma_{i-1}\in B.$$ Note that $\tau_i$ lies in $\mathcal{D}$, and is the simple element attached to $t_i$ in the Artin monoid. 
We now construct lifts of the generators of $N_k$ in $\widetilde{B}_0$. As $s_k^{(i)}\in W(r)$ for all $i=0, \dots, b_k-1$ we can take their lifts in $\mathcal{D}\subseteq B^+\subseteq B$ (obtained by lifting any reduced expression of them in $B$), which we denote by $\sigma_k^{(i)}$. As $s_k^{(i)}=\pi(\sigma_k^{(i)})\in U_0\subseteq N_0$, we have $\sigma_k^{(i)}\in \hat{B}_0$ for all $i=0, \dots, b_k-1$, and we can take their image in $\widetilde{B}_0$, which we still denote by $\sigma_k^{(i)}$.

Note that $s_k^{(0)}=t_{d_k} t_{d_k+1} \cdots t_{d_k+k-1}$ and we have $\ell(s_k^{(0)})=\sum_{i=0}^{k-1} \ell(t_{d_k+i})$, hence by Lemma~\ref{lift_rex} we have $\sigma_k^{(0)}=\tau_{d_k} \tau_{d_k+1}\cdots \tau_{d_k+k-1}$. 

\begin{lemma}\label{lem:power_d_garside}
Let $W = G(r,1,n)$ and $W_0$ a standard parabolic subgroup as above. Let $k$ be such that $b_k\neq 0$.
In $\widetilde{B}_0$, we have $$(\sigma_k^{(0)})^r= \Delta_{J_k^1}^2,$$ where $\Delta_{J_k^1}$ is the image in $\widetilde{B}_0$ of the simple element attached to the longest element $w_0(J_k^1)$ in the type $A$ parabolic subgroup of $W$ generated by $J_k^1$. In particular, if $k=1$, then $(\sigma_k^{(0)})^r=1$. 
\end{lemma}

\begin{proof}
We claim that in $\widetilde{B}_0$, for all $d_k\leq i \leq d_k+k-1$, we have $\tau_{d_k}^d=1$ and $$(\tau_i)^d=\sigma_{i-1} \sigma_{i-2} \cdots \sigma_{d_k+1} \sigma_{d_k}^2 \sigma_{d_k+1} \cdots \sigma_{i-1}$$ for $i>d_k$. Note that $$\tau_i=\sigma_{i-1} \cdots \sigma_{1} \sigma_{0} \sigma_{1}^{-1} \cdots \sigma_{i-1}^{-1}(\sigma_{i-1} \cdots \sigma_{2} \sigma_{1}^2 \sigma_{2}^{-1} \cdots \sigma_{i-1}^{-1})(\sigma_{i-1} \cdots \sigma_{3} \sigma_{2}^2 \sigma_{3}^{-1} \cdots \sigma_{i-1}^{-1})\cdots \sigma_{i-1}^2.$$
For all $j\leq d_k-1$, we have $\sigma_{i-1} \cdots \sigma_{j+1} \sigma_{j}^2 \sigma_{j+1}^{-1} \cdots \sigma_{i-1}^{-1}\in \QQ0$ (since the reflection $s_{i-1} \cdots s_j \cdots s_{i-1}$ is not in $W_0$ in that case, as the reflection $s_{d_k-1}$ appears in the reduced expression $s_{i-1} \cdots s_j \cdots s_{i-1}$). Hence deleting these factors, we get that 
\begin{eqnarray*}
\tau_i&=&\sigma_{i-1} \cdots \sigma_{1} \sigma_{0} \sigma_{1}^{-1} \cdots \sigma_{i-1}^{-1}(\sigma_{i-1} \cdots \sigma_{2} \sigma_{d_k}^2 \sigma_{2}^{-1} \cdots \sigma_{i-1}^{-1})(\sigma_{i-1} \cdots \sigma_{3} \sigma_{d_k+1}^2 \sigma_{3}^{-1} \cdots \sigma_{i-1}^{-1})\cdots \sigma_{i-1}^2\\
&=&\sigma_{i-1} \cdots \sigma_{1} \sigma_0 \sigma_1^{-1} \cdots \sigma_{d_k-1}^{-1} \sigma_{d_k} \cdots \sigma_{i-1}.
\end{eqnarray*}
Raising to the power $d$ we obtain 
$$(\tau_i)^d=\sigma_{i-1} \cdots \sigma_{1} \sigma_{0} (\underbrace{\sigma_1^{-1} \cdots \sigma_{d_k-1}^{-1} \sigma_{d_k} \cdots \sigma_{i-1}^2 \sigma_{i-2} \cdots \sigma_1}_{=:\beta} \sigma_0)^{d-1} \sigma_1^{-1} \cdots \sigma_{d_k-1}^{-1} \sigma_{d_k} \cdots \sigma_{i-1}.$$
We claim that $\beta\in \QQ0$. Indeed, we have 
\begin{eqnarray*}
\beta&=&\sigma_1^{-1} \cdots \sigma_{d_k-1}^{-1} \sigma_{d_k} \cdots \sigma_{i-1}^2 \sigma_{i-2} \cdots \sigma_1\\
&=&(\sigma_1^{-1} \cdots \sigma_{d_k-1}^{-1} \sigma_{d_k} \cdots \sigma_{i-1}^2 \sigma_{i-2}^{-1} \cdots \sigma_{d_k+1}^{-1} \sigma_{d_k} \cdots \sigma_1) \cdots (\sigma_1^{-1} \sigma_2^{-1} \cdots \sigma_{d_k-1}^{-1} \sigma_{d_k}^2 \sigma_{d_k-1} \cdots \sigma_1),
\end{eqnarray*}
and all the factors are in $\QQ0$ (since $s_{d_k-1}\notin W_0$). Hence this gives $$(\tau_i)^d=\sigma_{i-1} \cdots \sigma_{1} \underbrace{\sigma_{0}^d}_{\in \QQ0} \sigma_1^{-1} \cdots \sigma_{d_k-1}^{-1} \sigma_{d_k} \cdots \sigma_{i-1}$$ which lies in $\QQ0$ if $i=d_k$ and is equal to $\sigma_{i-1} \cdots \sigma_{d_k}^2 \sigma_{d_k+1}\cdots \sigma_{i-1}$ otherwise. This shows the claim. Now using the fact that the $\tau_i$'s commute with each other, we get that $$(\sigma_k^{(0)})^d=\sigma_{d_k}^2(\sigma_{d_k+1} \sigma_{d_k}^2 \sigma_{d_k+1} ) \cdots (\sigma_{d_k+k-2} \sigma_{d_k+k-3} \cdots \sigma_{d_k}^2 \sigma_{d_k+1}\cdots \sigma_{d_k+k-2})=\Delta_{J_k^1}^2,$$
where the last equality follows from the fact that in a type $A_m$ braid group with standard generators $\sigma_1, \sigma_2, \dots, \sigma_m$ (and corresponding Coxeter generators $s_1, s_2, \dots, s_m$), we have 

\begin{eqnarray*}
\Delta_{m-1}^2 \sigma_m \sigma_{m-1} \cdots \sigma_1^2 \sigma_2 \cdots \sigma_m&=&\Delta_{m-1} (\Delta_{m-1} \sigma_m \sigma_{m-1} \cdots \sigma_1) \sigma_1 \sigma_2 \cdots \sigma_m\\
&=&\Delta_{m-1} \Delta_m \sigma_1 \sigma_2\cdots \sigma_m=(\Delta_{m-1} \sigma_m \sigma_{m-1} \cdots \sigma_1) \Delta_m=\Delta_m^2, 
\end{eqnarray*}
where $\Delta_m$ is the lift of the longest element in $\langle s_1, s_2,\dots, s_m\rangle$ and $\Delta_{m-1}$ is the lift of the longest element of the parabolic subgroup $\langle s_1, s_2, \dots, s_{m-1} \rangle$. 

\end{proof}

We would like to show that our lifts $\sigma_k^{(i)}$ satisfy the defining relations of the Coxeter-like presentation of $U_0$. Unfortunately, this is not true, as it would give a splitting of the short exact sequence~(\ref{ses}). But we shall prove that the braid relations between these generators are satisfied already in $B$; the problem will come from the fact that the generators $\sigma_k^{(0)}$ fail to be of order $d$ inside $\widetilde{B}_0$, as we have just shown in Lemma~\ref{lem:power_d_garside}. We will need to twist our lift by a suitable element of $H_0$ to get the order relation in that case.   

\subsubsection{Braid relations}\label{subsub:braid} The fact that the lifts $\sigma_k^{(i)}$ satisfy the same braid relations as their images in $W$, that is, that
\begin{itemize}
\item $\sigma_k^{(0)}\sigma_k^{(1)}\sigma_k^{(0)}\sigma_k^{(1)}=\sigma_k^{(1)}\sigma_k^{(0)}\sigma_k^{(1)}\sigma_k^{(0)}$ for all $k$ such that $1 < b_k$,
\item $\sigma_k^{(i)} \sigma_k^{(i+1)} \sigma_k^{(i)} = \sigma_k^{(i+1)} \sigma_k^{(i)} \sigma_k^{(i+1)}$ for all $1 \leq i < b_k-1$,
\item $\sigma_k^{(i)}\sigma_k^{(j)}=\sigma_k^{(j)}\sigma_k^{(i)}$ for all $0 \leq i < j-1 \leq b_k-2$,
\item $\sigma_k^{(i)}\sigma_\ell^{(j)}=\sigma_\ell^{(j)}\sigma_k^{(i)}$ for all $k\neq \ell$, 
\end{itemize}

follows from the fact that if we consider these relations inside $U_0$ (that is, if we replace every $\sigma_k^{(i)}$ by $s_k^{(i)}$), then each side of each relation is an element of $W(r)$, and moreover, it is readily checked using~\cite[Lemma 1.5]{BREMKEMALLE} that the length of each side in terms of the generating set $S$ is the sum of the lengths of the various factors. For instance, we have 
$$\ell(s_k^{(0)}s_k^{(1)}s_k^{(0)}s_k^{(1)})=\ell(s_k^{(0)})+\ell(s_k^{(1)})+\ell(s_k^{(0)})+\ell(s_k^{(1)}).$$
Hence the fact that the relations above are satisfied in $B$ is a consequence of Lemma~\ref{lift_rex}.

\begin{remark}\label{rmq_braid_complement}
In fact, as $U_0$ is isomorphic to a complex reflection group which is a direct product of groups of type $G(r,1,b_k)$, we can consider its braid group $\mathcal{B}(U_0)$, which is a direct product of Artin groups of type $B_{b_k}$. What we did above is nothing but constructing a group morphism $\psi: \mathcal{B}(U_0)\rightarrow \widetilde{B}_0$. 
\end{remark}

\subsubsection{Order relations} \label{sect:orderrelsGr1n} Consider $s_k^{(i)}$ with $1 \leq i \leq b_k-1$. As $i \geq 1$ we have that $s_k^{(i)}\in\mathfrak{S}_n \cap N_W(G_0)$, where $G_0=W_0\cap \mathfrak{S}_n$. Hence we have $s_k^{(i)}\in N_{\mathfrak{S}_n}(G_0)$. The group $G_0$ is a (standard) parabolic subgroup of the type $A$ Coxeter group $\mathfrak{S}_n$ and $s_k^{(i)}$ is equal to $w_0(K_i) w_0(J_k^{i+1}\cup J_k^i)$ (see~\eqref{eq_si}), which is a generator of the Howlett complement $U^{G}_{0}$ to $G_0$ inside $N_{\mathfrak{S}_n}(G_0)$. 

In~\cite{NORMARTIN}, the definition of the Howlett complement $U^{G}_{0}$ was generalized to any reflection subgroup $G_0$ of a (possibly infinite) Coxeter group $G$. Let $\pi$ denote the quotient map $\mathcal{B}(G)\twoheadrightarrow G$. In the case where $G$ is finite, it was shown by Henderson and the authors~(see \cite[Corollary 3.12]{NORMARTIN}) that the set-theoretic map $U^{G}_0 \rightarrow \pi^{-1}(N_G(G_0))$ given by taking positive lifts of elements in $U^G_0\subseteq G$ in the Artin group $\mathcal{B}(G)$ of $G$ becomes an injective group homomorphism $U^G_0\rightarrow \pi^{-1}(N_G(G_0)) / \QQ0^G$ when passing to the quotient on the right side, where $\QQ0^G$ is the subgroup of the pure braid group of $G$ generated by the squares of the braided reflections around hyperplanes attached to reflections which are not in $G_0$. 

Applied to our case with $G=\mathfrak{S}_n$, where $\mathcal{B}(G)$ is the usual braid group $\mathcal{B}_n$ on $n$
strands, this result implies that the squares of the $\sigma_k^{(i)}$'s inside $\mathcal{B}_n$
belong to $\QQ0^G$. Indeed, they are images of $1=(s_k^{(i)})^2$ under the above injective morphism. Now, the morphism $\mathcal{B}_n \to B$ mapping the $j$-th Artin generator $\sigma_j$ for $1 \leq i \leq n-1$ to
$\sigma_j \in B$ is injective -- for instance because its composition with $B \to \mathcal{B}_{n+1}$, $\sigma_0 \mapsto \sigma_1^2$, $\sigma_j \mapsto \sigma_{j+1}$ for $j \geq 1$ is the injective morphism
$\mathcal{B}_n \to \mathcal{B}_{n+1}$, $\sigma_i \mapsto \sigma_{i+1}$ of adding one strand on the left. Moreover,
it maps any meridian around some reflecting hyperplane in $\mathfrak{S}_n$ (that is, a conjugate
of some squared generator of $\mathcal{B}_n$) to a meridian around the same reflecting hyperplane for $W$,
hence embeds $\QQ0^G$ inside $\QQ0$. Therefore $(\sigma_k^{(i)})^2\in \QQ0$ and, inside $\widetilde{B}_0$, we have $(\sigma_k^{(i)})^2=1$.

It remains to treat the case of the generators $s_k^{(0)}$, for all $k$ such that $b_k\neq 0$. We claim that for every such $k$, $\sigma_k^{(0)} T_k(\Delta_{J_k^1})$ is a lift of $s_k^{(0)}$ inside $\widetilde{H}_0$,
where $T_k \in \KK[X]$ is a polynomials satisfying the assumptions of Theorem \ref{theo:introGdeen}. Here, we abuse notation and write $\Delta_{J_k^1}$ for the image of the Garside element (i.e., the lift of the longest element of $\langle J_k^1\rangle$) of the type $A$ Artin group $\mathcal{B}(\langle J_k^1 \rangle) \subseteq B$ attached to $\langle J_k^1\rangle$ inside the Hecke algebra $H_0$ of $W_0$. As $\widetilde{H}_0$ is a quotient of the group algebra of $\widetilde{B}_0$ which contains the Hecke algebra $H_0$ of the parabolic subgroup $W_0$ (which contains $\langle J_k^1 \rangle$), this gives a well-defined element of $\widetilde{H}_0$.

Note that $\sigma_k^{(0)}$ commutes with every Artin generator $\sigma_i$ of $\mathcal{B}(\langle J_k^1\rangle)$, hence $\sigma_k^{(0)} \Delta_{J_k^1}=\Delta_{J_k^1} \sigma_k^{(0)}$. Let $T_k$ be a polynomial as in the assumptions
of Theorem \ref{theo:introGdeen}.
By Lemma~\ref{lem:power_d_garside}, we have $$(\sigma_k^{(0)} T_k(\Delta_{J_k^1}))^d=(\sigma_k^{(0)})^d T_k(\Delta_{J_k^1})^d= (\Delta_{J_k^1}^2) T_k(\Delta_{J_k^1})^d = 1$$

Note that $\sigma_k^{(i)}$ commutes with $\Delta_{J_k^1}$ whenever $2 \geq i$, hence the commutation relation between $T(\Delta_{J_k^1}) \sigma_k^{(0)}$ and $\sigma_k^{(i)}$ is satisfied inside $H_0\subseteq \widetilde{H}_0$ in that case. Also note that $\sigma_k^{(1)} \Delta_{J_k^j}=\Delta_{J_k^\ell} \sigma_k^{(1)}$, $\{j, \ell\}=\{1,2\}$ (as the images in $\mathfrak{S}_n$ of the various factors satisfy the same relation, with the length of the various factors adding). As $\Delta_{J_k^1}$ and $\Delta_{J_k^2}$ commute with each other and $\sigma_k^{(0)} \sigma_k^{(1)} \sigma_k^{(0)} \sigma_k^{(1)}=\sigma_k^{(1)} \sigma_k^{(0)} \sigma_k^{(1)} \sigma_k^{(0)}$ as seen in Subsection~\ref{subsub:braid}, we deduce that
\begin{align*}
\sigma_k^{(0)} T(\Delta_{J_k^1}) \sigma_k^{(1)} \sigma_k^{(0)} T(\Delta_{J_k^1}) \sigma_k^{(1)}&= T(\Delta_{J_k^1}) T(\Delta_{J_k^2}) \sigma_k^{(0)} \sigma_k^{(1)} \sigma_k^{(0)} \sigma_k^{(1)}\\&=T(\Delta_{J_k^1}) T(\Delta_{J_k^2}) \sigma_k^{(1)} \sigma_k^{(0)} \sigma_k^{(1)} \sigma_k^{(0)}\\&=T(\Delta_{J_k^2})T(\Delta_{J_k^1}) \sigma_k^{(1)} \sigma_k^{(0)} \sigma_k^{(1)} \sigma_k^{(0)}\\&=\sigma_k^{(1)} \sigma_k^{(0)} T(\Delta_{J_k^1})\sigma_k^{(1)} \sigma_k^{(0)} T(\Delta_{J_k^1}),
\end{align*}
hence the braid relation between $\sigma_k^{(0)} T(\Delta_{J_k^1})$ and $\sigma_k^{(1)}$ is still satisfied. This shows that one can lift $N_k$, for all $k$ such that $b_k\neq 0$. 

To conclude this case, note that the defined lifts of two generators belonging to different factors $N_p$ and $N_q$ of $U_0$ have to commute with each other. Indeed, on the one hand we have that $\Delta_{J_p^1}$ and $\Delta_{J_q^1}$ commute with each other for all $p,q$. On the other hand, $\Delta_{J_p^1}$ (resp. $\Delta_{J_q^1}$) also commutes with all $\sigma_q^{(i)}$ (resp. $\sigma_p^{(i)}$), and $s_p^{(i)}$, $s_q^{(j)}$ are elements in $W(r)$ with product also in $W(r)$ and such that ${s_p^{(i)}}\cdot {s_q^{(j)}}={s_q^{(j)}}\cdot {s_p^{(i)}}$ and $\ell({s_p^{(i)}}\cdot {s_q^{(j)}})=\ell({s_p^{(i)}})+\ell({s_q^{(j)}})$ for all $i,j$ and $p\neq q$, hence these observations together with Lemma~\ref{lift_rex} allow us to conclude that generators from distinct components $N_p$, $N_q$ commute with each other.  

Hence, we just showed that, modifying the group morphim $\psi$ from Remark~\ref{rmq_braid_complement} by twisting $\sigma_k^{(0)}$ by $\sigma_k^{(0)} T_k(\Delta_{J_k^1})$ for all $k$ such that $b_k\neq 0$, we get a group morphism $\psi'$ from $U_0\cong \overline{N_0}$ to $\widetilde{H}_0^\times$, which satisfies the assumptions of Lemma~\ref{lem:semidirect}. This concludes the proof of Theorem \ref{theo:introGdeen} in this case.

\section{Groups of type $G(de,e,n)$}
\label{sect:deen}

In this section, we complete the proof of Theorem \ref{theo:introGdeen}, using
the results of the preceeding section for $G(r,1,n)$ by setting $r = de$.

\subsection{General considerations}

Let $W = G(de,e,n)$, and $W^! = G(r,1,n)$ with $r = de$.
We denote by $\mathcal{A}$, $\mathcal{A}^!$ the corresponding
hyperplane arrangements and, for $L_0 \subset \C^n$, we denote by $W_0$ and $W_0^!$ the parabolic subgroups defined
as the pointwise stabilizers of $L_0$ in $W$ and $W^!$, respectively.
We have $W_0 = W \cap W_0^!$. 

Let $\mathcal{E}$ (resp. $\mathcal{E}^!$) denote the collection of all hyperplanes in $\mathcal{A}$ (resp. $\mathcal{A}^!$)
containing $L_0$. We can and will assume that
$L_0$ is the intersection of some collection of hyperplanes of $\mathcal{A}$,
which implies that $L_0$ is equal to the intersection $\bigcap \mathcal{E}$ of the hyperplanes inside $\mathcal{E}$ (and hence also that $L_0=\bigcap \mathcal{E}'$). 

Since $W_0$ (resp. $W_0^!$) is generated by its reflections, we have $$g \in N_0 := N_W(W_0)\Leftrightarrow g(\mathcal{E}) = \mathcal{E} \Leftrightarrow g(L_0) = L_0$$
and 
$$g \in N_0^!: = N_{W^!}(W_0^!) \Leftrightarrow g(\mathcal{E}^!) = \mathcal{E}^!.$$

Hence, as $g(L_0)=L_0$ implies that $g(\mathcal{E}^!)=\mathcal{E}^!$, we get $N_0\subset N_0^! \cap W$. Conversely, if $g \in N_0^! \cap W$
we get $g(\mathcal{E}^!)=\mathcal{E}^!$ hence $g(\bigcap \mathcal{E}^!)
= \bigcap \mathcal{E}^!$. But $\bigcap \mathcal{E} = L_0 =\bigcap \mathcal{E}^!$ hence $\bigcap \mathcal{E}^! = L_0$ and $g(L_0) = L_0$, whence $g \in N_0$. This proves the following:

\begin{lemma}\label{link_normalisateurs}
We have $N_0= N_0^! \cap W$.
\end{lemma}

At the level of braid groups, we denote by $B$ the braid group of $W$, $B^!$ the braid group
of $W^!$. Let $\pi : B \onto W$ and $\pi^! : B^! \to W^!$ be the natural
projections, with kernels $P,P^!$, respectively. Recall that $W$ is a normal subgroup of $W^!$ with cyclic quotient $C$ of order $e$.  

We set $M(\mathcal{A}) = \C^n \setminus \bigcup \mathcal{A}$, $M(\mathcal{A}^!) = \C^n \setminus \bigcup \mathcal{A}^!$. The inclusion 
map $M(\mathcal{A}^!) \subset M(\mathcal{A})$ induces a surjective
homomorphism $\pi_1(M(\mathcal{A}^!)) \onto \pi_1(M(\mathcal{A}))$.
Recall that $\hat{B}_0^! =  \pi_1(M(\mathcal{A}^!)/N_0^!)$, $\hat{B}_0 =  \pi_1(M(\mathcal{A})/N_0)$. 

We introduce $B^{\sharp} = \pi_1(M(\mathcal{A}^!)/W)$. It is a normal subgroup of $B^! = \pi_1(M(\mathcal{A}^!)/W^!)$, and there
is a natural projection map $B^{\sharp} \onto B$ with kernel $\Ker(M(\mathcal{A}^!) \onto M(\mathcal{A}))$.

 In particular we
have $B^{\sharp} = B$ when $\mathcal{A}^! = \mathcal{A}$, that is, as soon as $d > 1$. We denote by $\pi^{\sharp} : B^{\sharp} \onto W$
the natural projection map.

We denote by $p_W : W^! \to C = W^!/W$ the canonical projection and set $p_B = p_W \circ \pi^!$. We have $B^{\sharp} = \Ker p_B$ (see \cite{BMR} \S 3.B for a detailed description of the maps involved here and their properties). This yields the following commutative diagram: 
\begin{eqnarray}\label{diag:link}
\xymatrix{
 & 1 \ar[d] & 1 \ar[d] \\
 & P^! \ar[d] \ar@{=}[r] & P^! \ar[d] \\
 1 \ar[r] & B^{\sharp} \ar[d]_{\pi^{\sharp}} \ar[r] & B^! \ar[d]_{\pi^!} \ar[r]^{p_B} & C \ar@{=}[d] \ar[r] & 1 \\
 1 \ar[r] & W \ar[d] \ar[r] & W^! \ar[d] \ar[r]^{p_W} & C  \ar[r] & 1 \\
 & 1 & 1
} 
\end{eqnarray}

The usefulness of introducing this 'fake braid group' $B^{\sharp}$ is that it can be used
in general, regardless of the value of $d$, to define $\widetilde{B}_0$
as a
quotient of $\hat{B}_0^{\sharp} := (\pi^{\sharp})^{-1}(N_0)$. Recall
that $\widetilde{B}_0 = \hat{B}_0/\QQ0$ where $\QQ0$
is the kernel of the natural map $P = \pi_1(M(\mathcal{A})) \onto \pi_1(M(\mathcal{E}))$. Let us denote $q : P^! \onto P$ the projection
map orginating from the inclusion $M(\mathcal{A}^!) \subset M(\mathcal{A})$, and set $\QQ0^{\sharp} = q^{-1}(\QQ0)$. It is the
kernel of the projection map $P^! = \pi_1(M(\mathcal{A}^!)) \onto \pi_1(M(\mathcal{E}))$. 
We have the following commutative diagram
of short exact sequences
$$
\xymatrix{
1 \ar[r] & P^! \ar[r] \ar@{->>}[d]_q & \hat{B}_0^{\sharp} \ar[d] \ar[r] & N_0 \ar[r] \ar@{=}[d] & 1 \\
1 \ar[r] & P \ar[r] & \hat{B}_0 \ar[r] & N_0 \ar[r] & 1 
}$$
which implies that the middle map is also surjective, and that its
composite with the canonical projection $\hat{B}_0 \onto \widetilde{B}_0$
has kernel $q^{-1}(\QQ0) = \QQ0^{\sharp}$. This identifies $\widetilde{B}_0$ with $\hat{B}_0^{\sharp}/\QQ0^{\sharp}$.

From this new description we easily get the following slight
generalization of Lemma~\ref{lem:conjug} in the specific case of parabolic subgroups of $W=G(de,e,n)$:

\begin{lemma}\label{lem:conjug_big}
Let $W_1,W_2$ be two parabolic subgroups of $W=G(de,e,n)$,  which are conjugate in $W^!=G(de,1,n)$, let $G_i =N_W(W_i)$, and
$\hat{B}_i, \widetilde{B}_i, B_i,\widetilde{H}_i$ as in Lemma \ref{lem:conjug}.
Then there exists a group isomorphism 
$\widetilde{B}_1 \to \widetilde{B}_2$
mapping $B_1$ to $B_2$ and an algebra isomorphism $\widetilde{H}_1 \to \widetilde{H}_2$ mapping $H_1$ to $H_2$. 
\end{lemma}

\begin{proof}

We set $B^{\sharp}_i, \hat{B}^{\sharp}_i$ for $i = 1,2$
the groups $B^{\sharp},\hat{B}^{\sharp}$ defined above for $W_0 = W_i$.
We claim that $\hat{B}^{\sharp}_{1}$ and $\hat{B}^{\sharp}_{2}$ are conjugate in $B^!$.

This can be seen as follows: let $w\in W^!$ such that $w W_1 w^{-1}=W_2$. As $W\trianglelefteq W^!$ we have $w G_1 w^{-1}=G_2$. Let $b\in(\pi^!)^{-1}(\{ w\})$. Note that using the diagram~\eqref{diag:link}, we can see each $\hat{B}^{\sharp}_{i}$ inside $B^!$. We claim that $b \hat{B}^{\sharp}_{1} b^{-1}=\hat{B}^{\sharp}_{2}$. Indeed, for $x\in \hat{B}^{\sharp}_{1}$, we have that $\pi^!(b x b^{-1})\in w G_1 w^{-1}=G_2$. By~\eqref{diag:link}, it implies that there is $b'\in \hat{B}^{\sharp}_{2}$ such that $\pi^!(b')=\pi^!(b x b^{-1})$. Since both maps $\pi^{\sharp}$ and $\pi^!$ have the same kernel we have $\Ker\pi^!=P^!\subseteq \hat{B}^{\sharp}_{2}$ and since $b'\in \hat{B}^{\sharp}_{2}$, we get that $bxb^{-1}\in \hat{B}^{\sharp}_{2}$. Hence $b \hat{B}^{\sharp}_{1} b^{-1}\subseteq\hat{B}^{\sharp}_{2}$ and conversely we show that $b^{-1} \hat{B}^{\sharp}_{2} b\subseteq\hat{B}^{\sharp}_{1}$.   

The rest of the proof is then the same as in the proof of Lemma~\ref{lem:conjug}, only replacing $Q_i$ with $Q_i^{\sharp} = \Ker( \hat{B}^{\sharp}_i \onto \widetilde{B}_i)$, and noticing that it is the subgroup of $P^!$
generated by the meridians around the hyperplanes in $\mathcal{A}^!$
which are not reflecting hyperplanes for $W_i$.
\end{proof}

\subsection{Lifting complements}\label{lifting_complements}
From the commutativity of the diagram (\ref{diag:link}) we readily get that
 $\hat{B}^{\sharp}_0 = (\pi^{\sharp})^{-1}(N_0) = \hat{B}_0^! \cap B^{\sharp}$. Let us now consider the subgroup $\QQ0^{\sharp}$ of $P^!$, which is generated by
 the meridians around the hyperplanes which are not in $\mathcal{E}$.
 We have 
 $\widetilde{B}_0 = \hat{B}^{\sharp}_0 /\QQ0^{\sharp}$ and $\widetilde{B}_0^! = \hat{B}_0^!/\QQ0^!$. Let us denote $\hat{p}_B$ the
 restriction of $p_B$ to $\hat{B}_0^!$.
 Since $\QQ0^! \subset \Ker \hat{p}_B$, it induces a morphism $\widetilde{p}_B : \widetilde{B}_0^! \to C$
 whose kernel is exactly $\widetilde{B}_0^{\sharp} := \hat{B}_0^{\sharp}/\QQ0^!$, which projects onto $\widetilde{B}_0$ with kernel $\QQ0^{\sharp}/\QQ0^!$.

Let $U_0^!$ be a complement of $W_0^!$ inside $N_0^!$. We make the following assumption
\begin{eqnarray}\label{assumption_1}
U_0: = U_0^! \cap W\text{ is complementary to }W_0\text{ inside }N_0.
\end{eqnarray}
This assumption does \textit{not} hold in general, but one can always choose $U_0^!$ so that it works,
as we will see below.

Denote by $\widetilde{\pi}^! : \widetilde{B}_0^! \to W^!$, $\widetilde{\pi}^{\sharp} : \widetilde{B}_0^{\sharp} \to W$, $\widetilde{\pi} : \widetilde{B}_0 \to W$ the morphisms induced by $\pi^!$, 
$\pi^{\sharp}$ and $\pi$, respectively.
 Note that $\widetilde{p}_B=p_W\circ \widetilde{\pi}^!$. If we have a morphism $\psi^! : U_0^! \to \widetilde{B}^!_0$ satisfying the following assumption
\begin{eqnarray}\label{assumption_2}
\widetilde{\pi}^! \circ \psi^! = \Id_{U_0^!},
\end{eqnarray} 
then we also have $\widetilde{p}_B \circ \psi^! = (p_W)_{U_0^!}$. This implies that, under assumptions~(\ref{assumption_1}) and (\ref{assumption_2}), the map $\psi^!$ restricts to a morphism $\psi : U_0 \to \widetilde{B}_0^{\sharp} \subset \widetilde{B}_0^!$. Indeed, for $g \in U_0^!$, we have  $$\psi^!(g) \in  \widetilde{B}_0^{\sharp}~\Leftrightarrow~\widetilde{p}_B(\psi^!(g)) = 1~\Leftrightarrow~ p_W(g) = 1,$$
and the last condition is always fulfilled if $g\in U_0=U_0^! \cap W$. Finally, $\widetilde{\pi}^! \circ \psi^! = \Id_{U_0^!}$ immediately implies
$\widetilde{\pi}^{\sharp}\circ \psi = \Id_{U_0}$, so that $\psi$ indeed provides a convenient lift into $\widetilde{B}_0^{\sharp}$.

The complement $U_0^!$ 
that we chose is itself a complex reflection group,
with irreducible components belonging to the general series. 
Its braid group $\Gamma^!$ is an Artin group with irreducible components of type $B$,
with a projection map $\gamma^! : \Gamma^! \onto U_0^!$ (see Remark~\ref{rmq_braid_complement}). Let
$\Gamma = \{ g \in \Gamma^! \ ; \ \gamma^!(g) \in W \}$ and $\gamma$ the restriction of $\gamma^!$
to $\Gamma$.
Under assumption (\ref{assumption_1}),  $\gamma$ is a map $\Gamma \onto U_0$ and $\Ker(\Gamma^! \onto U_0^!)
= \Ker(\Gamma \onto U_0)$.

 While it is not possible,
in general, to obtain a lifting morphism $U_0^! \to \widetilde{B}^!_0$ as above, we were able in the cases satisfying this assumption to construct
morphisms $\psi^! : \Gamma^! \to \widetilde{B}_0^!$ such that 
$\widetilde{\pi}^! \circ \psi^! = \gamma^!$.

The restriction of $\psi^!$ to $\Gamma$ then provides a morphism $\psi : \Gamma \to  \widetilde{B}_0^{\sharp}$ such
that $\widetilde{\pi}^{\sharp} \circ \psi = \gamma$.

\subsection{Case where $W_0 = W_0^!$}

In some cases we can directly relate the algebras $\widetilde{H}_0^!$ and $\widetilde{H}_0$ for arbitrary
parameters. Recall that both $N_0^!$ and $N_0$
act on $\mathcal{E}$. In order to get comparable parameters
on both sides we need the following assumption
\begin{eqnarray}\label{assumption_3}
\mbox{The orbits of $\mathcal{E}$ under $N_0$ and $N_0^!$ are the same.}
\end{eqnarray} 

This condition will be satisfied in the cases in which we are interested.
Under this condition, one has a natural morphism from $\widetilde{H}_0$
to some specialization of $\widetilde{H}_0^!$. Indeed, recall that
$\widetilde{H}_0$ and $\widetilde{H}_0^!$ are both defined using parameters
$u_{s,i}, s \in \mathcal{R}_0^*, 0 \leq i < o(s)$, but with the additional
condition that $u_{s,i} = u_{wsw^{-1},i}$ for all $w \in N_0=N_W(W_0)$ in the case of $\widetilde{H}_0$,
 for all $w \in N_{W^!}(W_0)$ in the case of $\widetilde{H}_0^!$. Since the conjugation
 action of the normalizers on $\mathcal{R}_0^*$ is the same as their action on $\mathcal{E}$,
assumption (\ref{assumption_3}) says that the conditions on the parameters are the
same and thus the algebras are defined over the same
ring. 

Then, the composition
of $\KK \widetilde{B}_0^{\sharp} \to \KK \widetilde{B}_0^!$ and $\KK \widetilde{B}_0^! \to \widetilde{H}_0^!$
factorizes through $\KK \widetilde{B}_0^{\sharp} \to \widetilde{H}_0$,
as the Hecke relations are obviously mapped to $0$ inside $\widetilde{H}_0$,
and the possibly additional meridians around $\mathcal{A}^! \setminus \mathcal{A}$ are mapped to $1$. From this we get
 an
algebra morphism $\widetilde{H}_0 \to \widetilde{H}_0^!$.

The following lemma will be useful in the special case where $W_0 = W_0^!$.

\begin{lemma} \label{lem:specialW0W0} When $W_0= W_0^!$, the 
algebra morphism $\widetilde{H}_0 \to \widetilde{H}_0^!$ is injective.
\end{lemma}
\begin{proof} We first notice that the projection maps $\widetilde{B}_0^! \to N_0^!$  and
$\widetilde{B}_0^{\sharp} \to N_0$ 
induce isomorphisms 
$\widetilde{B}_0^!/B_0 \simeq N_0^!/W_0^! = N_0^!/W_0$
and a clearly surjective morphism
$\widetilde{B}_0^{\sharp}/B_0 \to N_0/W_0$. This morphism is actually an isomorphism,
as its kernel coincides with the kernel of the natural morphism $\widetilde{B}^{\sharp}_0/B_0 \to \widetilde{B}_0^!/B_0$,
which is injective as it is induced by the inclusion map $\hat{B}_0^{\sharp} \subset \hat{B}_0^!$.

We choose sets of representatives
 $E^{\sharp} \subset \widetilde{B}_0^{\sharp}$
and $E^{!} \subset \widetilde{B}_0^{!}$ such that $E^{\sharp} \subset E^!$. 
Then
we have morphisms of $\KK B_0$-modules
$$\bigoplus_{b \in E^{\sharp}} (\KK B_0) b \simeq \KK \widetilde{B}_0^{\sharp}  
\to \KK \widetilde{B}_0^! \simeq \bigoplus_{b \in E^!} (\KK B_0) b  
$$
where $\KK \widetilde{B}_0^{\sharp}  
\to \KK \widetilde{B}_0^!$ is the inclusion map, and the composition is
the identity on $E^{\sharp} \subset E^!$. Dividing out by the
defining ideals of $\widetilde{H}_0$ and $\widetilde{H}_0^!$ then
yields a sequence of morphisms
$$\bigoplus_{b \in E^{\sharp}} H_0 b \to \widetilde{H}_0  
\to \widetilde{H}_0^! \to \bigoplus_{b \in E^!} H_0 b  
$$
whose composite is injective. Now, the
morphisms $\bigoplus_{b \in E^{\sharp}} H_0 b \to \widetilde{H}_0$
and $\bigoplus_{b \in E^{!}} H_0 b \to \widetilde{H}_0^!$
are isomorphisms (see \cite{YH2} section 2.3.1), whence the natural morphism
$\widetilde{H}_0 \to \widetilde{H}_0^!$ is also injective.

\end{proof}

Notice that in this case, the natural morphism $\widetilde{B}_0^{\sharp} \to \widetilde{B}_0$
is actually an isomorphism.

\subsection{Lifting complements for standard parabolic subgroups}

Since the reflecting hyperplanes for $W$ are reflecting hyperplanes
for $W^!$, every parabolic subgroup $W_0$ of $W$ is a subgroup of a uniquely defined parabolic subgroup $W_0^!$
of $W^!$ of the same
rank. Now, 
by~\cite[Theorem 3.9]{TAYLORREFL}, 
every parabolic subgroup $W_0^!$ of $W^!$ is a conjugate inside $W^!$ of a standard
one. Hence by Lemma~\ref{lem:conjug_big}, we can assume that $W_0^!$ itself is standard. In this case $W_0$ is also standard. 

Recall from Section~\ref{sec:gr1n} that such a standard parabolic subgroup is determined by a pair $(n_0, \lambda)$, where $0\leq n_0\leq n$, and $\lambda$ is a partition of $n-n_0$. We will keep the notation introduced in Section~\ref{sec:gr1n}, except that all the groups related to $W^!$ will have the symbol $!$ as exponent. The pair $(n_0,\lambda)$ will be referred to as the \textit{type} of $W_0$. If $W_0$ has type $(n_0, \lambda)$, then 
$$W_0=G(de, e, n_0)\times\prod_{k=1}^{n} G(1,1,k)^{b_k},
\ \ W_0^!=G(de, 1, n_0)\times\prod_{k=1}^{n} G(1,1,k)^{b_k}.
$$
\subsubsection{Case where $n_0=0$} In this case,
we have $W_0=W_0^!$, and $U_0=W\cap U_0^!$ is a complement to $W_0$ inside $N_0$. The complement $U_0^!$ is a complex reflection group of type
$\prod_{k, b_k \neq 0 } G(de,1,b_k)$ with braid group an Artin group $\Gamma^!$
of type $\prod_{k, b_k \neq 0}  B_{b_k}$ (with the convention that $B_1 = A_1$ among Coxeter types).
The explicit description of $U_0$ is less straightforward than in
the other cases (see~\cite[Theorem 3.12]{TAYLORNORM}), however assumption (\ref{assumption_1}) is satisfied
and we constructed in Section~\ref{sec:gr1n} a morphism $\psi^! : \Gamma^! \to \widetilde{B}_0^!$ satisfying assumption (\ref{assumption_2}). Therefore we can apply the results of Section~\ref{lifting_complements} and by restriction we get a morphism
$\psi : \Gamma \to \widetilde{B}_0^{\sharp}$ such that $\widetilde{\pi}^{\sharp} \circ \psi = \gamma$. Since assumption (\ref{assumption_3}) is also satisfied,
the natural compositions towards the
Hecke algebras fit into the following commutative diagram (see the paragraph before Lemma~\ref{lem:specialW0W0}): 
$$
\xymatrix{
 \Gamma \ar[r] \ar[d] & \KK \widetilde{B}_0^{\sharp}\ar[d] \ar[r] & \widetilde{H}_0 \ar[d] \\
 \Gamma^! \ar[r] & \KK \widetilde{B}_0^! \ar[r] & \widetilde{H}_0^! \\
}
$$

Now by twisting $\psi^!$, we obtained in Section~\ref{sec:gr1n} (see the end of Subsection~\ref{sect:orderrelsGr1n}) a group morphism $(\psi')^!: \overline{N_0^!}\to (\widetilde{H}_0^!)^\times$ satisfying the assumptions of Lemma~\ref{lem:semidirect}. We claim that the restriction $\psi'$ of this map to $U_0$ provides a group morphism $U_0\cong \overline{N_0} \to \widetilde{H}_0^\times$ also satisfying the assumptions of Lemma~\ref{lem:semidirect}.   

Given $g\in\overline{N_0^!}$, an element $b\in \widetilde{B}_0^!$ as in the statement of Lemma~\ref{lem:semidirect} is (by the construction made in Section~\ref{sec:gr1n}) an element of the form $\psi^!(x)$ for some preimage $x$ of $g$ under $\gamma^!:\Gamma^! \to U_0^!$. In particular, if $g\in \overline{N_0}$, then $x$ lies in $\Gamma$ and $b=\psi^!(x)=\psi(x)\in \widetilde{B}_0^{\sharp}$. By the above commutative diagram and the fact that the map $\widetilde{H}_0 \to \widetilde{H}_0^!$ is injective, this implies that $b H_0\subseteq \widetilde{H}_0^!$ lies in fact inside $\widetilde{H}_0^\times$, and thanks to the isomorphism $\widetilde{B}_0^{\sharp} \to \widetilde{B}_0$ the element $b$ can be seen inside $\widetilde{B}_0$. This shows that the restriction of $\psi'$ to $\overline{N_0}$ satisfies the assumptions of Lemma~\ref{lem:semidirect}.    

\subsubsection{Case where $n_0\neq 0$} In this case, a complement $U_0$ to $W_0$ is obtained as the direct product of the subgroups $N_k$, where $N_k$ has the same generators as $N_k^!$ except for the first one, which is replaced by $s_0^{-k} s_k^{(0)}\in G(r,r,n)\subseteq W$, $r = de$ (see the proof of Theorem 3.12 of~\cite{TAYLORNORM}).  Note that
each $N_k$ is isomorphic to $N_k^!$ -- and to the complex reflection group $G(r,1,b_k)$ -- so that
$U_0$ is isomorphic to $U_0^!$, and we can also consider it as a quotient of
$\Gamma^!$ by the order relations corresponding to
the generators of $N_k$.

 These generators all have order $2$ except the first one which has order $r$.
In Section~\ref{sec:gr1n}, we described a lifting $\psi^! : \Gamma^! \to  \widetilde{B}_0^!$
such that $\widetilde{\pi}^! \circ \psi^! = \gamma^!$.
The images under $\psi^!$ of the standard generators
of $\Gamma^!$ were denoted $\sigma_k^{(i)}$.

The computations done in Section~\ref{sect:orderrelsGr1n} prove that, by replacing 
each $\sigma_k^{(0)}$
by 
$\sigma_0^{-k} c_k \sigma_k^{(0)}$
where $c_k = T_{0,k}(\sigma_0^e)T_k(\Delta_{J_{k,1}})\in \KK \widetilde{B}_0^!$ for some
polynomials $T_{0,k},  T_k \in \KK[X]$, we get another group
morphism $\psi' : \Gamma^! \to \KK \widetilde{B}_0^!$.

Moreover, notice that the image of each $c_k$
under the projection map $\KK \widetilde{B}_0^! \to \KK N_0^{!}$ lies inside $\KK N_0$. It follows that
the generators of $\Gamma^!$
have their image under $\psi'$ inside $\KK \widetilde{B}_0^{\sharp}$ and that $\psi'$
is actually a morphism $\Gamma^! \to \KK \widetilde{B}_0^{\sharp}$.

It remains to prove that we can choose the polynomials $T_{0,k}$ and  $T_k $ (for all $k$) so that the composition of the above morphism
with $\KK \widetilde{B}_0^{\sharp} \to \widetilde{H}_0$ factorizes
through some $\psi : U_0 \to \widetilde{H}_0$, making the following diagram commute
$$
\xymatrix{
\Gamma^! \ar[r]\ar[d] & \KK \widetilde{B}_0^{\sharp}\ar[d] \\
U_0 \ar@{.>}[r]  & \widetilde{H}_0 
}
$$
Indeed, the other conditions for applying Lemma \ref{lem:semidirect} are obviously satisfied
by such a $U_0 \to \widetilde{H}_0$, as each $c_k$ belongs to $H_0$ and we have a factorization $\KK \widetilde{B}_0^{\sharp} \to \KK \widetilde{B}_0 \to \widetilde{H}_0$.

To this end, as the braid relations are still satisfied by the modified generators, as well as the order relations except possibly for the $\sigma_0^{-k} c_k \sigma_k^{(0)}$, the only missing condition is the order relation $1 = (\sigma_0^{-k} c_k \sigma_k^{(0)})^{de}$. In other words, we need to have the equality 
$$1 = (\sigma_0^{-k} c_k \sigma_k^{(0)})^{de} = \sigma_0^{-k de}c_k^{de} (\sigma_k^{(0)})^{de}$$ inside $\widetilde{H}_0$. By Lemma~\ref{lem:power_d_garside}, this is
equal to 
$\sigma_0^{-k de}c_k^{de} \Delta_{J_k^1}^2$. Hence we need to have $c_k^{de} = \sigma_0^{k de} \Delta_{J_k^1}^{-2}$ inside $H_0 \subseteq  \widetilde{H}_0$.
Therefore, it is sufficient to have polynomials $T_{0,k},  T_k \in \KK[X]$ such that
$T_{0,k}(\sigma_0^e)^{de} =  \sigma_0^{k de} = (\sigma_0^e)^{kd}$ and $T_k(\Delta_{J_{k,1}})^{de} = \Delta_{J_k^1}^{-2}$ inside $\widetilde{H}_0$,
and this provides the conditions of Theorem \ref{theo:introGdeen}.

\begin{remark}

In the particular case where $n_0\neq 0$ and $\la =1^{n-n_0}$, we have that $\widetilde{B}_0^!$ is the semidirect product of $B_0^!$ with $U_0^!$, as shown in~\cite[Proposition 5.1]{NORMARTIN} (the product is even direct in that case). In this case, the only $k$ for which $b_k\neq 0$ is $1$, and in this case $\Delta_{J_k^1}=1$ as it is the Garside element in a braid group on one strand (note that Lemma~\ref{lem:power_d_garside} applies to this case, and provides the order relation on $\sigma_1^{(0)}$ which allows the splitting of the sequence (\ref{ses})). Hence in this case, we only need to find one polynomial $T_{0,1}$ such that $T_{0,1}(\sigma_0^e)^{de}=(\sigma_0^e)^d$, as the constant polynomial $T_1=1$ satisfies $T_1(\Delta_{J_1^1})^{de}=\Delta_{J_1^1}^{-2}$. 

\end{remark}

\section{Parabolic subgroups of maximal rank in exceptional types}

We assume here that the parabolic subgroup $W_0$ of the exceptional (irreducible) reflection group $W$
has rank $\rk W - 1$. Using the Shephard-Todd notation, we prove the following.
\begin{theorem} \label{theo:maxipars} Let $W$ be an irreducible complex reflection group of exceptional type,
and $W_0$ a parabolic subgroup of maximal rank. Let $z_{B_0}$ be the canonical
positive central element of $B_0$. Assume that the pair $(W,W_0)$ is not of type $(G_{33},D_4)$
or $(G_{25},(\mathbb{Z}/3\mathbb{Z})^2)$. If there exists $T \in \KK[X]$ such that $T(z_{B_0})^{|Z(W)|} = z_{B_0}^{-|Z(W_0)|}$
inside $H_0$, then $\widetilde{H}_0 \simeq \overline{N_0} \ltimes H_0$. In the two exceptional cases, the same conclusion
holds with the condition replaced by $T(z_{B_0})^3 = z_{B_0}^{-1}$.

\end{theorem}

In the case of rank $2$, all maximal parabolic subgroups have rank $1$, and $z_{B_0}$
is the braided reflection associated to the unique distinguished reflection inside $W_0$.
By Lemma \ref{lem:polfield} an immediate consequence of the theorem is the following.

\begin{corollary} Assume that $W$ is an irreducible exceptional complex reflection group of rank $2$, and $W_0$ a parabolic
subgroup of rank $1$ and order $m$. Then $H_0$ has parameters $u_0,\dots,u_{m-1} \in \KK^{\times}$.
Assume that $i \neq j \Rightarrow u_i - u_j \in \KK^{\times}$
and that each $u_i$ admits an $r$-th root inside $\KK$, for
$r = |Z(W)|$. Then $\widetilde{H}_0 \simeq \overline{N_0} \ltimes H_0$.
\end{corollary}

According to \cite{TAYLORNORM} Theorem~5.5, we have $N_0 = W_0 \times Z(W)$ in almost all cases.
In these cases, we have $z_B, z_{B_0} \in \hat{B}_0$ and 
$z_B^{|Z(W)|} z_{B_0}^{-|Z(W_0)|} \in \QQ0$. Therefore, it
is sufficient to find $T \in \KK[X]$ such that $T(z_{B_0})^{|Z(W)|} = z_{B_0}^{-|Z(W_0)|}$ inside $H_0$,
so that $z_B T(z_{B_0})$ has order $|Z(W)|$.

We now consider the exceptions. There are two exceptions in rank $2$, for $W \in \{ G_{13}, G_{15} \}$,
and the ones in higher rank 
are for $W = G_{35} = E_6$, which is already known by \cite{NORMARTIN} since $W$ is a Coxeter group in this case, and $W \in \{ G_{25}, G_{33} \}$. We deal with these 
 cases now. For the presentations and group-theoretic properties of $B$ we are going to use freely, we refer to \cite{BANNAI} for the groups of rank $2$, to \cite{BMR,BESSISMICHEL,BESSISKPI1} for the other ones.

\subsection{$W = G_{13}$}

When $W$ has type $G_{13}$, a presentation of $W$ (see \cite{BMR}) is given by generators $s,t,u$ and relations 
$tust=ustu$, $stust=ustus$, $s^2 = t^2 = u^2 = 1$. The only case when $N_0 \neq Z(W).W_0$ is when $W_0$ is a conjugate of $\langle s \rangle$. In this case $Z(W).W_0$ has index $2$ inside $N_0$, and $\overline{N_0}$ is cyclic. Moreover $Z(W)$ is cyclic of order $4$ generated by $z = (stu)^3$. A presentation of $B$
is given by generators $\sigma,\tau,\upsilon$ and
relations 
$\tau\upsilon\sigma\tau=\upsilon\sigma\tau\upsilon$, $\sigma\tau\upsilon\sigma\tau=\upsilon\sigma\tau\upsilon\sigma$ and a generator of the center is $z_B = (\sigma \tau \upsilon)^3$. The group $B = B_{13}$ is isomorphic to the Artin group of type $I_2(6)=G_2$ with presentation $\langle a,b \ | \ ababab = bababa\rangle$, an isomorphism being given by $b \mapsto \upsilon$, $ a \mapsto (\upsilon \sigma \tau \upsilon)^{-1}$
 with inverse $\upsilon \mapsto b$, $\tau \mapsto a^{-1}ba$, $\sigma \mapsto \Delta^{-1} a^2$
with $\Delta = ababab$. Since $\Delta$ is central,
we have $a \in \hat{B}_0$. 
One checks that $\pi(a)$ has order $6$ inside $N_0$ and generates $\overline{N_0}$.
Moreover, $\Delta = z_B^{-1}$. In order to lift $\pi(a)$,
we look for polynomials $T \in \KK[X]$ such that $(aT(\sigma))^8 = 1$ inside $H_0$. We have
$$
(aT(\sigma))^8 
= (a^2)^4T(\sigma)^8
= \sigma^4 \Delta^{-4}T(\sigma)^8
= \sigma^4 z_B^{4}T(\sigma)^8
= \sigma^4 \sigma^{2}T(\sigma)^8
= \sigma^6 T(\sigma)^8
$$
since $z_B^4 \sigma^{-2} \in \QQ0$, and we need to find a polynomial $T$ satisfying $T(\sigma)^8 = \sigma^{-6}$ inside $H_0$. For this it is enough to get one such that $T(\sigma)^4 = \sigma^{-3}$. Since $r = |Z(W)| = 4$
this follows from Lemma \ref{lem:polfield}.

\subsection{$W = G_{15}$}

When $W$ has type $G_{15}$, a presentation of $W$ is given by generators 
$s_1,s_2,s_3$ and relations 
$s_1s_2s_3=s_3s_1s_2$,  $s_2s_3s_1s_2s_1=s_3s_1s_2s_1s_2$,
$s_1^2 = s_2^3 = s_3^2 = 1$.
The only case when $N_0 \neq Z(W).W_0$ is when $W_0$ is a conjugate of $\langle s_3 \rangle$. In this case $Z(W).W_0$ has index $2$ inside $N_0$, and $\overline{N_0}$ is cyclic. The braid group has a presentation
with generators braided reflections $\sigma_1,\sigma_2,
\sigma_3$ and relations
$\sigma_1\sigma_2\sigma_3=\sigma_3\sigma_1\sigma_2$, $\sigma_2\sigma_3\sigma_1\sigma_2\sigma_1=\sigma_3\sigma_1\sigma_2\sigma_1\sigma_2$ and $\pi$ maps $\sigma_i$ to $s_i$. It embeds inside the Artin group $A = \langle \sigma,\tau \ | \ \sigma \tau \sigma \tau = \tau \sigma\tau\sigma \rangle$
of type $B_2$ under $\sigma_3 \mapsto \tau^4$,
$\sigma_1 \mapsto \sigma$, $\sigma_2 \mapsto \tau \sigma \tau^{-1}$, and can be identified in this way with the kernel of $A \onto \Z/4\Z$ given by $\tau \mapsto 1$, $\sigma \mapsto 0$. Under this identification, 
the positive generator of $Z(B)$ is $z_B = \sigma_2\sigma_3\sigma_1\sigma_2\sigma_1=\sigma_3\sigma_1\sigma_2\sigma_1\sigma_2$ and maps to $(\sigma \tau\sigma\tau)^2$. We notice that 
$\alpha = \sigma_1\sigma_2\sigma_3=\sigma_3\sigma_1\sigma_2$ maps
to $\tau^2 (\tau \sigma \tau \sigma)$, which centralizes $\tau^4 = \sigma_3$. The order of $\pi(\alpha)$ is $24$
inside $N_0$ and in the quotient group $\overline{N_0}$ as well. Therefore $N_0 = W_0 \rtimes \langle \pi(\alpha)\rangle$. Notice that $\alpha^2 = z_B \tau^4 = z_B \sigma_3$.

We then look for polynomials $T \in \KK[X]$ such that
$(\alpha T(\sigma_3))^{24} = 1$ inside $H_0$. We have
$$
(\alpha T(\sigma_3))^{24} 
= \alpha^{24} T(\sigma_3)^{24}
= z_B^{12} \sigma_3^{12} T(\sigma_3)^{24}
= \sigma_3^{14} T(\sigma_3)^{24}
$$
since $|Z(W)| = 12$ hence $z_B^{12}\sigma_3^{-2} \in \QQ0$
and we need a polynomial $T$ satisfying $T(\sigma)^{24}= \sigma_3^{-14}$.
For this it is enough to have $T \in \KK[X]$ with $T(\sigma)^{12}= \sigma_3^{-7}$,
and this is again a consequence of Lemma \ref{lem:polfield} under our assumptions since $r = 12$.

\subsection{$W = G_{25}$, $W_0 = (\Z/3\Z)^2$}
\label{sect:G25wC3C3}

We have that $B_{25}$ is isomorphic to the Artin group of type $A_3$. We denote its standard Artin generators by $\sigma_1,\sigma_2,\sigma_3$. The Hecke algebra relation is $(\sigma_i-u_0)(\sigma_i-u_1)(\sigma_i-u_2)=0$. We have $B_0 = \langle \sigma_1,\sigma_3 \rangle$.

We have that $\overline{N_0}$ is cyclic of order $6$, generated by the image of $(\sigma_1\sigma_2\sigma_3)^2$. Now, $(\sigma_1\sigma_2\sigma_3)^4 = z_B$, and $z_{B_0} = \sigma_1 \sigma_3$ is centralized by $(\sigma_1\sigma_2\sigma_3)^2$.
Moreover, conjugation by $(\sigma_1 \sigma_2 \sigma_3)^2$ exchanges $\sigma_1$ with $\sigma_3$ (and in particular
does \emph{not} centralize $B_0$). This has for consequence that, although $W_0$ is not irreducible, we
have only $3$ parameters and the Hecke algebra relation is $(\sigma_i-u_0)(\sigma_i-u_1)(\sigma_i-u_2)=0$ for $i \in \{1,3\}$. Since $|Z(W)| = 3$, we have that $z_B^3$ and $z_{B_0}^3$ are the full loops
for $W$ and $W_0$, respectively, so that $z_B^3 = z_{B_0}^3$ inside $\widetilde{B}_0$.
For any $T \in \KK[X]$ we have
$$
\left( (\sigma_1\sigma_2\sigma_3)^2 T(z_{B_0}) \right)^6 = z_{B_0}^3 T(z_{B_0})^6
.$$
It is therefore enough to find $T \in \KK[X]$ with $T(z_{B_0})^2 = z_{B_0}^{-1}$ inside $H_0$.
Now, one checks that
$z_{B_0}$ is annihilated inside $H_0$ by the polynomial $\prod_{0 \leq i,j \leq 2} (X - u_iu_j)$. By Lemma \ref{lem:polfield} such a $T$ exists as soon as this
polynomial is square-free and $\KK$ contains a square root of each of the 
$u_iu_j$, $0 \leq i,j \leq 2$, and this is the case as soon as $\KK$ contains the $\sqrt{u_i}$.

\subsection{$W = G_{33}$, $W_0 = D_4$}

A presentation of $B_{33}$ is with generators
$s,t,u,v,w$, Artin relations symbolized by the diagram
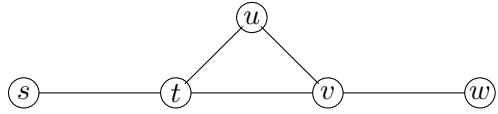
\begin{center}

\begin{tikzpicture}
\draw (0,0) circle (0.2);
\draw (0,0) node {$s$};
\draw (2,0) circle (0.2);
\draw (2,0) node {$t$};
\draw (4,0) circle (0.2);
\draw (4,0) node {$v$};
\draw (6,0) circle (0.2);
\draw (6,0) node {$w$};
\draw (3,1) circle (0.2);
\draw (3,1) node {$u$};
\draw (0.2,0) -- (1.8,0);
\draw (2.2,0) -- (3.8,0);
\draw (4.2,0) -- (5.8,0);
\draw (2.12,.12) -- (2.88,0.88);
\draw (3.88,.12) -- (3.12,0.88);
\end{tikzpicture}
\end{center}
and the additional relations $(vtu)^2 = (uvt)^2=(tuv)^2$
obtained in \cite{BESSISMICHEL,BESSISKPI1} and implemented
in (the development version of) CHEVIE, see \cite{CHEVIE}. A presentation of $B_{34}$ is deduced from it by adding another generator commuting with all the other ones except $w$, and satisfying an Artin relation of length $3$ with it.

The center of $B_{33}$ is generated by $z_B = (stuvw)^5$.
We let $x = stut^{-1}s^{-1}$. Then, $t,v,w,x$ satisfy the Artin relations of type $D_4$ -- as in easily checked using CHEVIE -- and provide generators for the braid group of a parabolic subgroup $W_0$ of type $D_4$.

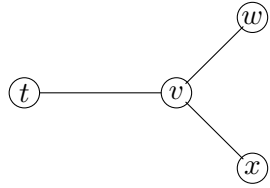
\begin{center}

\begin{tikzpicture}
\draw (0,0) circle (0.2);
\draw (0,0) node {$t$};
\draw (2,0) circle (0.2);
\draw (2,0) node {$v$};
\draw (3,1) circle (0.2);
\draw (3,1) node {$w$};
\draw (3,-1) circle (0.2);
\draw (3,-1) node {$x$};
\draw (0.2,0) -- (1.8,0);
\draw (2.12,.12) -- (2.88,0.88);
\draw (2.12,-.12) -- (2.88,-0.88);
\end{tikzpicture}
\end{center}
The group $\overline{N_0}$ is cyclic of order $6$, and contains the center of $W$, which has order $2$. Let $c_1 = t(stvwuvtu)^{-1}$. One checks that $\bullet \mapsto \bullet^{c_1}$ permutes the generators of the $D_4$ diagram clockwise. Moreover, its image inside $\overline{N_0}$ has order $6$. Now, explicit
computations show that $z_B^2 c_1^6 = b^{12}$ where
$b = txvw$. Then, considering the generator $z_{B_0}$
of the Artin group $B_0$ of type $D_4$, since $W_0$ has Coxeter number $6$ and center of order $2$ we have $b^6 =  z_{B_0}^2$, whence $c_1^6 = z_{B_0}^4 z_B^{-2} = (z_{B_0}/z_B)^4 z_B^2$.
Since both $Z(W)$ and $Z(W_0)$ have order $2$, $z_{B_0}^2$ and $z_B^2$ are the full loops
in their respective hyperplane complements, and therefore $(z_{B_0}/z_B)^2$ is mapped to $1$
in $\widetilde{H}_0$, and $z_B^2$ is mapped to $z_{B_0}^2$. We thus need to find $T \in \KK[X]$ such that
$(c_1T(z_{B_0}))^6 = z_{B_0}^2 T(z_{B_0})^6 =1$ inside $H_0$, so it is enough to have $T(z_{B_0})^3 = z_{B_0}^{-1}$.

Using CHEVIE we check that, inside $H_0$, $z_{B_0}$ is annihilated by the polynomial
$$
(X-u_0^8u_1^4)(X- u_0^9u_1^3)(X- u_0^{12})(X-  u_0^6u_1^6)(X+u_0^6u_1^6)(X-  u_0^4u_1^8)(X-  u_0^3u_1^9)(X-  u_1^{12})$$
where $u_0,u_1$ are the eigenvalues of the Artin generators. Notice that the condition
that this polynomial is square-free implies that $u_0^6u_1^6 - (- u_0^6u_1^6) = 2 u_0^6u_1^6 \in \KK^{\times}$,
hence the characteristic of $\KK$ cannot be $2$.

This concludes the proof of Theorem \ref{theo:maxipars}.

\section{On the remaining exceptional types}
\label{sect:remaining}

Assume that $W$ is of exceptional type, and $W_0$ is not maximal (and non-trivial). 
In particular $W$ has rank at least $3$. Since the case where $W$ is
a real reflection group is known by Theorem \ref{theo:liftreal}, there only remain $9$ exceptional types to consider. 

 The problem of determining possible liftings
using a systematic computer search fails in general for a couple of reasons, one of
them being the following one. The only known ways so far to
solve the word problem for the complex braid groups are, either to embed them into some Artin group of finite Coxeter type when this is possible, or
to use methods from Garside theory, which involve the monoids (or categories) introduced
by Bessis in~\cite{BESSISKPI1}. One of the problems with these monoids is that there is no known method
yet to write an element of the pure braid group as a product of natural generators
in 1-1 correspondence with the (distinguished) reflections -- even worse, it seems that no such collection of generators has ever been determined,
and the minimal number of generators of the pure braid group is greater than the number of atoms of the monoid. Therefore,
one cannot hope to mimic the methods of Digne and Gomi in \cite{DIGNEGOMI} for the real case. Finally, the more direct approach given by applying a generic Reidemeister-Schreier method from the morphism $B \onto W$ also
fails most of the time, because the groups $W$ are quite big.

In this section we are nevertheless able to deal with all the rank $3$ groups except $G_{27}$,
and with $G_{32}$ (rank $4$). This leaves five open cases ($G_{27}$,$G_{29}$,$G_{31}$,$G_{33}$,$G_{34}$). Notice that, when $W_0$ has rank $1$ (which is the only case to consider when $W$ has rank $3$)
then the centralizer of $W_0$
is equal to its normalizer (see~\cite[Lemma 5.1]{TAYLORNORM}).

\subsection{$W = G_{25}$}
\label{sect:G25rang1}

The braid group $B$ of $W$ is the Artin group of type $A_3$,
with generators $\sigma_1,\sigma_2,\sigma_3$ and the (ordinary) braid
relations between them. Then $W$ is its quotient
by the relations $\sigma_i^3 = 1$, and its Hecke algebra is defined by
the relations $(\sigma_i - u_0)(\sigma_i-u_1)(\sigma_i-u_2)=0$.
 There is a single class of distinguished reflections.

We let $W_0 = \langle s_1 \rangle$. 
A complement of $W_0$ inside $N_0$ is given
by $U_0 = \langle s_3, \zeta_1 \rangle \simeq G(3,1,2)$ with $\zeta_1 = (s_1s_2)^3$. The defining relations
are $s_3^3 = \zeta_1^2 = 1$ and $s_3 \zeta_1 s_3 \zeta_1=  \zeta_1s_3 \zeta_1 s_3$.

 We have $\sigma_3^3 \in \QQ0$. We set $z_1 = (\sigma_1 \sigma_2)^3$. We check that
$z_1$ and $\sigma_3$ satisfy a braid relation of type $B_2$.
Then, $z_1^2$ is the full loop inside $\langle \sigma_1, \sigma_2 \rangle$,
hence $z_1^2 \equiv \sigma_1^3 \mod \QQ0$. If there exists $T \in K[X]$ such
that $T(\sigma_1)^2 = \sigma_1^{-3}$,
which is the case
by Lemma \ref{lem:polfield} as soon as
the  $u_i - u_j \in \KK^{\times}$ for $i \neq j$ and
$\KK$ contains $\sqrt{u_i}$, 
 since $\sigma_1$ commutes with $z_1$
we get a lift $z_1 T(\sigma_1)$ of order $2$ which still satisfies the
$B_2$ relation with $\sigma_3$. Since the image of $\sigma_3$ has order $3$ this
provides a convenient lift of $U_0 \simeq \overline{N_0}$.

\subsection{$W = G_{26}$}

The braid group of $W$ is the Artin group of type $B_3$,
with generators indexed as follows
\begin{center}
\begin{tikzpicture}
\draw (0,0+.05) -- (1,0+.05);
\draw (0,0-.05) -- (1,0-.05);
\draw (1,0) -- (2,0);
\fill[color=cyan] (0,0) circle (.3);
\fill[color=cyan] (1,0) circle (.3);
\fill[color=cyan] (2,0) circle (.3);
\draw (0,0) node {$\sigma_1$};
\draw (1,0) node {$\sigma_2$};
\draw (2,0) node {$\sigma_3$};
\end{tikzpicture}
\end{center}
and the Hecke relations are $(\sigma_1-v_0)(\sigma_1 - v_1) = 0$, $(\sigma_i-u_0)(\sigma_i - u_1)(\sigma_i - u_2) = 0$ for $i \in \{2,3 \}$.
Then $W$ is the quotient of $B$ by the relations $\sigma_1^2 = \sigma_2^3 = \sigma_3^3 = 1$, and $s_i$ denotes the
image of $\sigma_i$ inside $W$. 
\subsubsection{$W = G_{26}$, $W_0 = \Z/3\Z$}
\label{sect:G26wC3}

We first look at $W_0 = \langle s_3 \rangle \simeq \Z/3\Z$. Then $\overline{N_0}$
has order $36$, and a complement of $W_0$ inside $N_0$
is obtained via $U_0 = \langle z_1,z_W,s_1 \rangle$,
where $z_1 = (s_2s_3)^3$, $z_W = (s_1s_2s_3)^3$.
The orders of $s_1,z_1,z_W$ are $2,2,6$ and the subgroup generated by $s_1,z_1$ is a dihedral group of order $12$,
which contains $z_W^3=(s_1z_1)^3$. 
It is easily checked that $U_0$ has for presentation
$$
\langle s_1,z_1,z_W \ | \ s_1^2 = z_1^2 = 1, (s_1z_1)^3 = (z_1 s_1)^3 = z_W^3, s_1z_W = z_W s_1, z_1 z_W = z_W z_1  \rangle.
$$

If we manage to find $\tilde{s}_1 \in s_1 H_0,\tilde{z}_1 \in z_1 H_0,\tilde{z}_W \in z_W H_0 \subset \widetilde{H}_0$ with the same orders and satisfying these relations, we are
done.

We choose $\tilde{s}_1 = \sigma_1$, $\tilde{z}_W = z_BT_1(\sigma_3)$, $\tilde{z}_1 = (\sigma_2 \sigma_3)^3T_2(\sigma_3)$ for some $T_1,T_2 \in \KK[X]$
and with $z_B = (\sigma_1\sigma_2\sigma_3)^3 \in Z(B)$.
Then the commutation relations of $\tilde{z}_W$ are satisfied, and $\sigma_1^2 \in \QQ0$ hence $\tilde{s}_1$ has order $2$ inside $\widetilde{H}_0$. 
Moreover, $z_B^6$ is the central full loop hence is equal to $\sigma_3^3$ modulo $\QQ0$, and $((\sigma_2 \sigma_3)^3)^2$
is the central full loop  for the parabolic subgroup $\langle s_2,s_3 \rangle$ of type $G_4$, hence is also equal
to $\sigma_3^3$ modulo $\QQ0$. It follows that, if $T_1,T_2$ can be chosen so that $T_1(\sigma_3)^6 = \sigma_3^{-3}$ and $T_2(\sigma_3)^2 = \sigma_3^{-3}$
then these lifts have the appropriate orders.  For $T_1$ it is enough to have $T'_1 \in \KK[X]$ such that
$T'_1(\sigma_3)^2 = \sigma_3^{-1}$. By Lemma \ref{lem:polfield} a sufficient condition for these
polynomials to exist is that $\prod_i(X-u_i)$ is square-free and $\KK$ contains a square root of each $u_i$.

Therefore the only thing
remaining to be checked is that $(\tilde{s}_1 \tilde{z}_1)^3(\tilde{z}_1\tilde{s_1})^{-3}$ is $1 \in H_0 \subset \widetilde{H}_0$. Since $T_1(\sigma_3)$ and $T_2(\sigma_3)$ commute with the other terms involved, it is enough to check that $x = (\sigma_1 (\sigma_2 \sigma_3)^3)^3((\sigma_2 \sigma_3)^3\sigma_1)^{-3}
\in \hat{B}_0$ actually belongs to $\QQ0$.

For this we do computations in the computer system GAP4, using a Reidemeister-Schreier type method to get a generating set for $P = \Ker(B \onto W)$ and express $x$ (as a lengthy expression) in terms of these generators. The generators obtained by this method are 21 elements which turn out to be conjugates of powers
of the generators. These are the following ones.

$$
\begin{array}{l}
\sigma_1^{-2}, ^{\sigma_2}(\sigma_1^{-2}), \sigma_2^3,(\sigma_1^{-2})^{\sigma_2},\sigma_3^3, ^{\sigma_1}(\sigma_2^3),^{\sigma_2}(\sigma_3^3),(\sigma_3^3)^{\sigma_2},
^{\sigma_3\sigma_2}(\sigma_1^{-2}),(\sigma_1^{-2})^{\sigma_2\sigma_3^{-1}},
(\sigma_1^{-2})^{\sigma_2\sigma_3},^{\sigma_1\sigma_2}(\sigma_3^3),\\
(\sigma_3^3)^{\sigma_2\sigma_1^{-1}},
^{\sigma_2\sigma_1\sigma_2}(\sigma_3^3),
^{\sigma_2^{-1}\sigma_1\sigma_2}(\sigma_3^3),
^{\sigma_1^2\sigma_3^{-1}\sigma_2}(\sigma_1^{-2}),
^{\sigma_2\sigma_1^2\sigma_3^{-1}\sigma_2}(\sigma_1^{-2}),
^{\sigma_1^{-1}\sigma_2^2(\sigma_2\sigma_1)^2\sigma_2^{-1}}(\sigma_3^3),\\
(\sigma_1^{-2})^{\sigma_2^{-1}\sigma_3(\sigma_2^{-1}\sigma_1^{-1})^2\sigma_1^{-1}\sigma_2\sigma_1\sigma_2^2(\sigma_2\sigma_1)^2\sigma_2^{-1}\sigma_3^3
\sigma_2(\sigma_1^{-1}\sigma_2^{-1})^2\sigma_2^{-2}\sigma_1}, 
(\sigma_3^3)^{\sigma_2\sigma_1^{-1}\sigma_3^3\sigma_2},
(\sigma_3^3)^{\sigma_2\sigma_1^{-1}(\sigma_2\sigma_3^2)^2\sigma_3\sigma_2^{-1}\sigma_3^{-3}}
\end{array}
$$

Being conjugates of powers of the generators, they belong to
$\QQ0 = \Ker(P \onto P_0 \simeq \Z)$ exactly when the corresponding
hyperplane is not $\Ker(s_3 - 1)$. This is the case for all of them
except for $\sigma_3^3$, which is mapped to $1$ under $P \onto P_0 \simeq \Z$. Using GAP4 we get that the image of $x$ under this map is $0$,
and this defines a morphism $U_0 \to \widetilde{H}_0^{\times}$ satisfying the assumption of Lemma \ref{lem:semidirect}.

\subsubsection{$W = G_{26}$, $W_0 = \Z/2\Z$}
\label{sect:G26wA1}

We let $W_0 = \langle s_1 \rangle \simeq \Z/2\Z$.
In this case $\overline{N_0}$ has order $72$, and a complement $U_0$ of $W_0$ inside
$N_0$ is easily checked to
be generated by $s_3$ and $z_3 = s_1s_2s_1s_2$. These elements
both have order $3$, and satisfy the relations $(s_3 z_3)^2 = (z_3 s_3)^2$. These relations are known to be defining relations
of the complex reflection group $G_5$, which has order $72$, therefore they indeed provide a presentation of $U_0$.

It is thus sufficient to find elements $\tilde{s}_3 \in s_3 H_0$,$\tilde{z}_3 \in z_3 H_0$ satisfying these relations. We set $\tilde{s}_3 = \sigma_3$ and $\tilde{z}_3 = (\sigma_1\sigma_2)^2 T(\sigma_1)$ for
$T \in \KK[X]$ such that $\tilde{z}_3$ has order $3$.
This is possible when $(\sigma_1\sigma_2)^6T(\sigma_1)^3$
is $1$ inside $H_0$. Now, since $(\sigma_1\sigma_2)^2$ is the positive generator of the center of the parabolic
braid group associated to the parabolic subgroup
$\langle s_1,s_2 \rangle$ of type $G(3,1,2)$, whose center has order $3$, we
know that $(\sigma_1\sigma_2)^6$
is the corresponding full central loop. Therefore, its image inside
$P_0$ is the full loop associated to the parabolic
subgroup $W_0$, that is $\sigma_1^2$. It follows that, inside $\widetilde{H}_0$, we have
 $(\sigma_1\sigma_2)^6T(\sigma_1)^3 = \sigma_1^2T(\sigma_1)^3$, and we want $T$ to satisfy
$T(\sigma_1)^3 = \sigma_1^{-2}$ inside $H_0$. 
By Lemma \ref{lem:polfield} this is possible
as soon as $(X - v_0)(X - v_1)$ is square-free and $\KK^{\times}$ contains 3rd roots of the parameters $v_i$.
 Then, it
is sufficient to check that $x = (\sigma_3 (\sigma_1\sigma_2)^2)^2
((\sigma_1\sigma_2)^2\sigma_3)^{-2} \in \QQ0$ to get a convenient morphism $U_0 \to \widetilde{H}_0^{\times}$, and this
is checked by the same computational method as in section \ref{sect:G26wC3}.

\subsection{$W = G_{24}$}
\label{sect:G24}
\subsubsection{\it A new presentation for $G_{24}$}

The presentation given in CHEVIE, originating from \cite{BESSISMICHEL}, is by generators $s,t,u$, and relations
$stst = tsts$, $sus=usu$, $tut = utu$ and $stustustu
= tstustust$. We propose an alternative presentation
such that the centralizer of a reflection  is easily described. We introduce additional generators $x = s^t = t^{-1}st$, $y = ^s t = s t s^{-1}$. Using CHEVIE it is easily checked
that the relations symbolized by the following diagram hold.
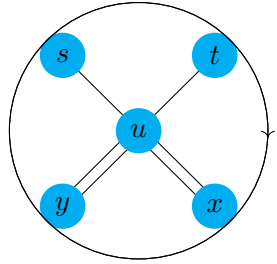
\begin{center}
\begin{tikzpicture}
\draw (0,0) -- (1,1);
\draw (0,0) -- (-1,1);

\draw (0+.05,0+.05) -- (1+.05,-1+.05);
\draw (0-.05,0-.05) -- (1-.05,-1-.05);
\draw (0+0.05,0-0.05) -- (-1+0.05,-1-0.05);
\draw (0-0.05,0+0.05) -- (-1-0.05,-1+0.05);

\fill[color=cyan] (1,1) circle (.3);
\fill[color=cyan] (1,-1) circle (.3);
\fill[color=cyan] (-1,1) circle (.3);
\fill[color=cyan] (-1,-1) circle (.3);
\fill[color=cyan] (0,0) circle (.3);
\draw[decoration = {markings, mark=at position 0 with {\arrow{<}}},
  postaction={decorate}] (0,0) circle (1.7);
\draw (0,0) circle (1.7);
\draw (-1,1) node {$s$};
\draw (1,1) node {$t$};
\draw (1,-1) node {$x$};
\draw (-1,-1) node {$y$};
\draw (0,0) node {$u$};
\end{tikzpicture}
\end{center}

In this `steering wheel' diagram, all edges represent Artin relations, that is $sus=usu$, $tut = utu$, $xuxu = uxux$, $uyuy = yuyu$, and the oriented circle has the same meaning
as for the Corran-Picantin presentations of the groups
$G(e,e,n)$ (see \cite{CORPIC}), namely it symbolizes the relation
$st=tx=xy=ys$, originating from the dual braid monoid of dihedral type $I_2(4) = B_2$.

In order to prove that this indeed provides a presentation
of $B$, it is sufficient to check that the relation $stustustu
= tstustust$ is a consequence of these relations. This is done as follows~:
$$
\begin{array}{lclclcl}
stustustu&=& stu ( ^s t) \underline{sus}tu &=&
stuyus\underline{utu} 
&=& stuyustut \\ &=& stuyu ( ^s t) sut 
&=& st\underline{uyuy} sut &=& styuy\underline{usu}t \\
&=& \underline{st}yuysust & =& t \underline{xy} u\underline{ys}ust 
& = & tstustust \\
\end{array}
$$

A collection of 21 generators for $P = \Ker(B \onto W)$ is obtained
as follows

$$
\begin{array}{l}
s^{-2}, t^{-2}, u^{-2}, x^{-2}, y^{-2}, ^s(u^{-2}), ^t(u^{-2}), ^u(x^{-2}), ^u(y^{-2}), ^x(u^{-2}), ^y(u^{-2}), ^{st}(u^{-2}), 
  ^{su}(x^{-2}),\\ ^{su}(y^{-2}), ^{sx}(u^{-2}), ^{sy}(u^{-2}), ^{tu}(x^{-2}), ^{tu}(y^{-2}), 
  ^{xu}(y^{-2}), ^{yu}(x^{-2}), ^{usx}(u^{-2})
  \end{array}
$$

\subsubsection{The centralizer of a reflection}

Let $W_0 = \langle \pi(u) \rangle$. Then $\overline{N_0}$ has order $8$,
and it is easily checked that $W_0$ admits a complement
$U_0$ inside $N_0$ generated by $\pi(xux)$ and $\pi(yuy)$. Actually,
both elements have order $2$ and satisfy a Coxeter relation
of length 4. In order to lift this complement it is thus enough to find elements $\tilde{a},\tilde{b} \in B$ mapping to $\pi(xux),\pi(yuy)$, such that $\tilde{a}^2,\tilde{b}^2, (\tilde{a}\tilde{b})^2(\tilde{b}\tilde{a})^{-2} \in \QQ0$.

By simply setting $\tilde{a} = xux$,
$\widetilde{B} = yuy$, it is checked computationally that,
when written in terms of the generators of $P$ given above,
these elements map to $0$ under the morphism $P \onto P_0 \simeq \Z$ where all the generators map to $0$ but $u^{-2}$. In the case of the order relation, it is even
possible to check this `by hand', as
$$
u^{-2}.ux^2u^{-1}.u^2.x^2.x^{-1}u^2 x
= u^{-1}x^2uxu^2 x
= u^{-1}x(xuxu)u x
= (xuxu)u^{-1}xu x = (xux)^2
$$
and similarly with $x$ replaced with $y$. This provides a group-theoretic splitting in this case.

\subsection{$W = G_{32}$}
\label{sect:G32}

In this case, $B$ is the Artin group of type $A_4$, with generators $\sigma_1,\dots,\sigma_4$.

\subsubsection{$W = G_{32}$, $W_0 = \mathbb{Z} / 3 \mathbb{Z}$}

Let $W_0= \langle s_1 \rangle$. By computer we get $|\overline{N_0}|
= 1296 = | G_{26}|$. We set $z_1 = (\sigma_1 \sigma_2)^3$.
Then, the images of $z_1,\sigma_3,\sigma_4$ inside $W$
centralize $W_0$, have order $2,3,3$, and generate together a subgroup $U_0$ of order $1296$ intersecting $W_0$
trivially. Therefore it is a complement to $W_0$ inside $N_0$.

Moreover, inside $B$ one checks that $z_1,\sigma_3,\sigma_4$ satisfy the relations of the braid group of $G_{26}$, which also has order $1296$. Since the order of the generators for $G_{26}$ are also $2,3,3$ this proves
in particular that $U_0$ is isomorphic to $G_{26}$.  
Finally, we have $\sigma_3^3  \in \QQ0$, $\sigma_4^3 \in \QQ0$, and $z_1^2 \equiv \sigma_1^3\mod \QQ0$. Letting $x = z_1 T(\sigma_1)$, $y = \sigma_3$, $z = \sigma_4$ with $T(\sigma_1)^2 = \sigma_1^{-3}$, we get $x^2 = y^3 = z^3= 1$ and that $x,y,z$ satisfy inside $\widetilde{H}_0$ the braid relations of the braid group of type $G_{26}$
as soon as $\prod_i(X-u_i)$ is square-free and $\KK$ contains the $\sqrt{u_i}$, by Lemma~\ref{lem:polfield}.

\subsubsection{$W = G_{32}$, $W_0 = \mathbb{Z} / 3\mathbb{Z} \times \mathbb{Z} / 3 \mathbb{Z}$}

We first consider the case where $ W_0 = \langle s_1,s_3 \rangle$. Then $\overline{N_0}$ has order $72$
and permutes non-trivially the two distinguished reflections in $W_0$. Therefore we still have only 3 parameters
$u_0,u_1,u_2 \in \KK^{\times}$ for the Hecke algebra of $W_0$.
We
set $a = (s_1 s_2 s_3)^2$, $c = (s_3 s_4)^3$. Then
$\langle a,c \rangle$ has trivial intersection with $W_0$,
normalizes it, and has order $72$, like $G(6,1,2)$.
Moreover, $a$ and $c$ have order $6$ and $2$, respectively.
We let $\tilde{a} = (\sigma_1 \sigma_2 \sigma_3)^2$ and
$\tilde{c} = (\sigma_3 \sigma_4 )^3$. These are preimages
of $a,c$, and one checks that $(\tilde{a}\tilde{c})^2 = 
(\tilde{c}\tilde{a})^2$
inside the braid group $B$. In particular, we have that $\overline{N_0}\cong G(6,1,2)$. Now, $\tilde{a}^6$ is the full central twist
of the parabolic subgroup $\langle s_1,s_2,s_3 \rangle$,
and thus its image inside $\widetilde{H}_0$ is equal to
the image of $z_1 = (\sigma_1\sigma_3)^3$. Similarly,
the image of $\tilde{c}^2$ is equal to the image of $z_2 = \sigma_3^3$.

\medskip

We check that $z_1,z_2,z_2^{\tilde{a}}$ and $\tilde{c}$ pairwise commute,
that $\tilde{a}^2$ commutes with $z_1,z_2$,
and that $\tilde{a}$ commutes with $z_1$.
Letting
$x = \tilde{a}T_x(z_1)$, $y = \tilde{c} T_y(z_2)$
for some $T_x,T_y \in \KK[X]$, we have 
$$
\begin{array}{lcl}
xyxy &=&  
\tilde{a}T_x(z_1)\tilde{c} T_y(z_2) \tilde{a}T_x(z_1)\tilde{c} T_y(z_2) \\
&=&  
\tilde{a}\tilde{c}\tilde{a}\tilde{c}T_x(z_1) (T_y(z_2)^{ \tilde{a}})T_x(z_1) T_y(z_2) \\
&=&  
\tilde{c}\tilde{a}\tilde{c}\tilde{a}  (T_y(z_2)^{\tilde{a}^2}) T_x(z_1)(T_y(z_2)^{ \tilde{a}})T_x(z_1) \\
&=&  
\tilde{c}\tilde{a}  (T_y(z_2)^{\tilde{a}}) T_x(z_1)\tilde{c}\tilde{a}(T_y(z_2)^{ \tilde{a}})T_x(z_1) \\
&=&  
\tilde{c}  T_y(z_2)\tilde{a} T_x(z_1)\tilde{c}T_y(z_2)\tilde{a}T_x(z_1) \\
&=&  
yxyx \\
\end{array}
$$

Moreover, we have $x^6 = \tilde{a}^6 T_x(z_1)^6= z_1 T_x(z_1)^6$,
while $y^2 = \tilde{c}^2 T_y(z_2)^2 = z_2 T_y(z_2)^2$,
so we need to have $T_x(z_1)^6 = z_1^{-1}$,
$T_y(z_2)^2 = z_2^{-1}$, and
$z_1$ and $z_2$ are obviously 
annihilated
by 
$\prod_{0 \leq i,j \leq 2 } (X - u_i^3 u_j^3)$
and $\prod_i (X - u_i^3)$. In order to apply Lemma \ref{lem:polfield}
we thus need that these polynomials are square-free and that
$\sqrt{u}_i \in \KK , 0 \leq i \leq 2$.

\subsubsection{$W = G_{32}$, $W_0 = G_4$}
\label{sect:G32G4}

We now consider the case where $ W_0 = \langle s_1,s_2 \rangle \simeq G_4$. Then $\overline{N_0}$ again has order $72$, but is not isomorphic to $G(6,1,2)$. Note that $G(6,1,2)$ and $G_5$ are two complex reflection groups of the same rank and the same order which are not isomorphic as abstract groups.

We
set $a = (s_1 s_2 s_3)^4$, $c = s_4$. They both centralize $W_0$,
have order $3$, satisfy $(ac)^2 = (ca)^2$, and $\langle a,c \rangle$ has order $72$ and has trivial intersection with $W_0$. 
Therefore it is a suitable complement, isomorphic to the complex reflection group $G_5$. Let
$\tilde{a} = (\sigma_1 \sigma_2 \sigma_3)^4$, $\tilde{c} = \sigma_4$. One checks that $(\tilde{a}\tilde{c})^2 = (\tilde{c}\tilde{a})^2$ inside $B$, and moreover $\tilde{c}^3 \in \QQ0$, $\tilde{a}^3 \equiv z_{B_0}^2 \mod \QQ0$ where $z_{B_0} = (\sigma_1 \sigma_2)^3$. If there exists $T \in \KK[X]$ with $z_{B_0}^2T(z_{B_0})^3 = 1$,
letting $x = \tilde{a}T(z_{B_0})$,
 and
$y = \tilde{c}$ provides a morphism $\overline{N_0} \to \widetilde{H}_0^{\times}$ to which we can apply Lemma \ref{lem:semidirect}.

Now, the eigenvalues of $z_{B_0}$ can be computed on the irreducible representations of the
Hecke algebra of $G_4$ (see e.g. \cite{MARINWAGNER} \S 5.3).
One gets that $z_{B_0}$ acts on the irreducible representations of the generic Hecke algebra
of type $G_4$ with the values $\{ u_i^3 , -u_i^3u_j^3, (u_0u_1u_2)^2 ; 0 \leq i,j \leq 2, i \neq j \}$,
and therefore it is annihilated inside $H_0$ by the polynomial
\begin{equation}
\label{eq:polminG4}
\left(
\prod_{0 \leq i \leq 2} (X - u_i^3)
\right)
\left(
\prod_{\stackrel{0 \leq i,j \leq 2}{i \neq j}} (X +u_i^3u_j^3)
\right)
\left(
X - (u_0u_1u_2)^2
\right)
\end{equation}

Therefore, Lemma 2.6 can be applied when this polynomial is square-free 
and $(u_0u_1u_2)^2$ admits a 3rd root, which is in particular the case when $u_0u_1u_2$ admits such a root.

\subsection{Conditions for $G_{24}$, $G_{25}$, $G_{26}$, $G_{32}$}

\begin{proposition} Let $W = G_{24}$, and $u_0,u_1 \in \KK^{\times}$ the defining parameters. 
Let
$W_0$ be a parabolic subgroup of
$W$. Then $\widetilde{H}_0 \simeq \overline{N_0} \ltimes H_0$
in the following cases:
\begin{enumerate}
\item If $W_0$ has rank $1$; in this case, there is a group-theoretic splitting and the complement
is a Coxeter group of type $B_2$.
\item If $W_0$ has rank $2$ and type $A_2$, when the polynomial $(X-u_0^6)(X-u_1^6)(X+u_0^3u_1^3)$ is square-free
and $-u_0^3u_1^3$ admits a square root in $\KK$.
In this case the
complement is $Z(W) \simeq \mathbb{Z} / 2 \mathbb{Z}$.
\item If $W_0$ has rank $2$ and type $B_2$; in this case, there is a group-theoretic splitting and the complement is $Z(W) \simeq \mathbb{Z} / 2 \mathbb{Z}$.
\end{enumerate}
\end{proposition}
\begin{proof}
In the case where $W_0$ has rank $1$, this has been proved in Section~\ref{sect:G24}. When $W_0$ has rank $2$ it is maximal
and the complement is $Z(W) \simeq \mathbb{Z} / 2 \mathbb{Z}$. Therefore, when $W_0$ has type $B_2$, since $Z(W_0) \simeq \mathbb{Z} / 2 \mathbb{Z}$,
we get a group-theoretic splitting via $\langle z_B z_{B_0}^{-1}\rangle$. And, when $W_0$ has type $A_2$, since $Z(W_0) = 1$,
we get an isomorphism $\widetilde{H}_0 \simeq \overline{N_0} \ltimes H_0$ by Lemma~\ref{lem:polfield} under the condition of the statement, by
Table~\ref{tab:tabulationDeltaAn} (see the proof of Lemma \ref{lem:typea} and the comment after it on how to determine the polynomial from the table).

\end{proof}

\begin{proposition} Let $W = G_{25}$, and $u_0,u_1,u_2 \in \KK^{\times}$ the defining parameters.
 Let
$W_0$ be a parabolic subgroup of
$W$. Then $\widetilde{H}_0 \simeq \overline{N_0} \ltimes H_0$ in the following cases:
\begin{enumerate}
\item If $W_0$ has rank $1$, when the polynomial $\prod_i (X-u_i)$ is square-free and
 $\sqrt{u_i} \in \KK$. In this case the
complement is isomorphic to a complex reflection group of type $G(3,1,2)$.
\item If $W_0$ has rank $2$ and type $G_4$,
when the polynomial (\ref{eq:polminG4}) is square-free and
 $\sqrt[3]{u_0u_1u_2} \in \KK$.
 In this case the
complement is $Z(W) \simeq \mathbb{Z} / 3 \mathbb{Z}$.
\item If $W_0$ has rank $2$ and type $\mathbb{Z} / 3 \mathbb{Z}\times \mathbb{Z} / 3 \mathbb{Z}$, when the polynomial 
$\prod_{0 \leq i,j \leq 2} (X - u_iu_j)$ is square-free and
$\sqrt{u_i} \in \KK$. 
In this case the
complement is a cyclic group of order $6$.
\end{enumerate}
\end{proposition}
\begin{proof}
In the case where $W_0$ has rank $1$, this has been proved in Section~\ref{sect:G25rang1}. When $W_0$ has rank $2$ and
type $\mathbb{Z} / 3\mathbb{Z} \times \mathbb{Z} / 3\mathbb{Z}$, this has been proved in Section~\ref{sect:G25wC3C3}. We now assume
that $W_0$ is maximal of type $G_4$. Then $|Z(W)|=3$ and $|Z(W_0)|=2$ and
the polynomial (\ref{eq:polminG4}) annihilates $z_{B_0}$
as in Section~\ref{sect:G32G4}, and we get similarly the
additional condition $\sqrt[3]{u_0u_1u_2} \in \KK$.
This concludes the proof.
\end{proof}

\begin{proposition} Let $W = G_{26}$ and
 $v_0,v_1,u_0,u_1,u_2\in \KK^{\times}$
 the defining parameters.
The Hecke relations are $(\sigma_1-v_0)(\sigma_1-v_1)=0$,
$(\sigma_i-u_0)(\sigma_i-u_1)(\sigma_i-u_2)=0$ for $i=2,3$. Let
$W_0$ be a parabolic subgroup of
$W$. Then $\widetilde{H}_0 \simeq \overline{N_0} \ltimes H_0$ in the following cases:
\begin{enumerate}
\item If $W_0$ has rank $1$ and type $\mathbb{Z} / 2 \mathbb{Z}$, when $v_1 - v_0 \in \KK^{\times}$ and $\sqrt[3]{v_i} \in \KK$. In this case the
complement is isomorphic to a complex reflection group of type $G(3,1,2)$.
\item If $W_0$ has rank $1$ and type $\mathbb{Z} / 3 \mathbb{Z}$, when $\prod_i(X-u_i)$ is square-free and $\sqrt{u_i} \in \KK$. In this case the
complement is isomorphic to a complex reflection group of type $G(6,2,2)$.
\item If $W_0$ has rank $2$ and type $\mathbb{Z} / 2 \mathbb{Z} \times \mathbb{Z} / 3\mathbb{Z}$; in this case, we have a group-theoretic splitting,
and the
complement is $Z(W)$ which is cyclic of order $6$.
\item If $W_0$ has rank $2$ and type $G_4$, when the polynomial (\ref{eq:polminG4}) is square-free and $\sqrt[3]{u_0u_1u_2} \in \KK$. In this case the
complement is $Z(W)$ which is cyclic of order $6$.
\item If $W_0$ has rank $2$ and type $G(3,1,2)$, when
the polynomial
$$
\left( \prod_{(i,j) \in \{0,1\}\times \{ 0,1,2 \}} (X - v_i^2u_j^2) \right)\left(\prod_{i \neq j} (X + v_0v_1 u_i u_j) \right)
$$
is square-free and 
 $\sqrt{-v_0v_1 u_i u_j} \in \KK$ for $i \neq j$. In this case the
complement is $Z(W)$ which is cyclic of order $6$.
\end{enumerate}
\end{proposition}
\begin{proof} The proof of (1) is given in Section~\ref{sect:G26wA1}.
The proof of (2) is given in Section~\ref{sect:G26wC3}, and the
identification of the complement with the group $G(6,2,2)$ is done
using GAP4 algorithms for identifying the isomorphism type of small groups.
We now consider the case of maximal parabolic subgroups. We have $|Z(W)| = 6$,
$|Z(G(3,1,2))| = 3$, $|Z( G_4)| = 2$. When $W_0 = \mathbb{Z} / 2\mathbb{Z} \times \mathbb{Z} / 3\mathbb{Z}$, since $|Z(W)| = |Z (W_0)|$, we get a group-theoretic splitting.
In case $W_0$ has type $G_4$, we have $z_{B_0} = (\sigma_2 \sigma_3)^3$, 
and we know by Section~\ref{sect:G32G4} that it is annihilated by polynomial (\ref{eq:polminG4}),
and for applying Lemma \ref{lem:polfield} we need to be able to take 3rd roots of its
roots,
that is, we need to have $\sqrt[3]{u_0u_1u_2} \in \KK$.

 When $W_0$ has type $G(3,1,2)$, we have $z_{B_0} = \sigma_1 \sigma_2 \sigma_1 \sigma_2$. By computing
its value on the irreducible representations of the generic Hecke algebra of $G(3,1,2)$, we get that it is annihilated
by the polynomial of the statement.
In order to apply Lemma \ref{lem:polfield} we thus need this polynomial to be square-free, and also
to have square roots of its roots, that is,
we need to have $\sqrt{-v_0v_1 u_i u_j} \in \KK$.
\end{proof}

\begin{proposition} Let $W = G_{32}$, and $u_0,u_1,u_2 \in \KK^{\times}$ the
defining parameters.
Let $W_0$ be a parabolic subgroup of
$W$. Then $\widetilde{H}_0 \simeq \overline{N_0} \ltimes H_0$ in the following cases:
\begin{enumerate}
\item If $W_0$ has rank $1$, when $\prod_i (X-u_i)$ is square-free and
 $\sqrt{u_i} \in \KK$. In this case the
complement is isomorphic to a complex reflection group of type $G_{26}$.
\item If $W_0$ has rank $2$ and type $G_4$, when
the polynomial (\ref{eq:polminG4}) is square-free and
 $\sqrt[3]{u_0u_1u_2} \in \KK$.
 In this case the
complement is isomorphic to a complex reflection group of type $G_5$.

\item If $W_0$ has rank $2$ and type $\mathbb{Z} / 3\mathbb{Z} \times \mathbb{Z} / 3\mathbb{Z}$, when 
$\prod_{0 \leq i,j \leq 2 } (X - u_i^3 u_j^3)$
and $\prod_i (X - u_i^3)$ are square-free 
and
$\sqrt{u_i} \in \KK$, $0 \leq i \leq 2$. In this case the
complement is isomorphic to a complex reflection group of type $G(6,1,2)$.
\item If $W_0$ has rank $3$ and type $G_4 \times \mathbb{Z} / 3\mathbb{Z}$, then there is a group-theoretic splitting and the complement
is $Z(W)$ which is cyclic of order $6$.
\item If $W_0$ has rank $3$ and type $G_{25}$, when 
the polynomial
$$
\left(\prod_{\{ i,j, k \} = \{0,1,2 \}} (X - u_i^2u_j^4 u_k^6)
(X - u_i^3u_j^3 u_k^6)
 (X - u_i^4 u_j^4 u_k^4)\right) $${}$$
 \times 
\left(\prod_{i \neq j} (X - u_i^4u_j^8)(X-u_i^6 u_j^6)\right)
\left(\prod_{i} (X - u_i^{12})\right)
\left( X^3 - (u_0u_1u_2)^{12} \right)
$$
is split and square-free (which implies that $\mu_3(\C) \subset \KK$), and we have $\sqrt{u_i} \in \KK$, $0 \leq i \leq 2$.
In this case the
complement is $Z(W)$ which is cyclic of order $6$.
\end{enumerate}
\end{proposition}
\begin{proof}
The cases where $W_0$ has rank $1$ or $2$ have been dealt with in
Section~\ref{sect:G32}; we can thus assume that $W_0$ is a maximal parabolic subgroup
of rank $3$. In this case, we need to find a polynomial $P$ such that $P(z_{B_0})^{|Z(W)|} = z_{B_0}^{|Z(W_0)|}$. 
We first consider the case where $W_0 = G_4 \times \mathbb{Z} / 3\mathbb{Z}$.

Then $|Z( W_0)| = 6 = |Z( W)|$,
and we get a group-theoretic splitting $Z(W_0) \to \widetilde{B}_0$ via $\zeta_6 \mapsto z_B z_{B_0}^{-1}$
with $z_{B_0} = (\sigma_1 \sigma_2)^3 \sigma_4$.

We then  consider the case where $W_0$ has type $G_{25}$. Then
$|Z( W_0)|=3$, $|Z( W)|=6$, and from the explicit computation of the
values of $z_{B_0}$ on the irreducible representations of the generic Hecke algebra of type $G_{25}$ (see e.g. \cite{MARINWAGNER} \S 5.3) we get that $z_{B_0}$ is annihilated by the polynomial
of the statement. Therefore the condition of Lemma~\ref{lem:polfield} is that this
polynomial is split and square-free in $\KK$, and 
$\KK$ must contain
the square roots of all its roots. This translates into the conditions of the statement.

\end{proof}

\end{document}